\documentclass[11pt,reqno]{amsart}
\usepackage{tikz}
\usepackage{subfig}
\usepackage{epstopdf}
\usepackage{epsfig}
\usepackage{math}
\graphicspath{{./figures/}}
\usepackage{tikz}
\usetikzlibrary{shapes,arrows}
\usepackage{adjustbox}
\usepackage{rotating}
\usepackage{graphicx}
\usepackage{algorithm}
\usepackage{algpseudocode}

\tikzstyle{decision} = [diamond, draw, fill=blue!20, 
    text width=4.5em, text badly centered, node distance=3cm, inner sep=0pt]
\tikzstyle{block} = [rectangle, draw, fill=blue!20, 
    text width=5em, text centered, rounded corners, minimum height=4em]
\tikzstyle{line} = [draw, -latex']
\tikzstyle{cloud} = [draw, ellipse,fill=red!20, node distance=3cm,
    minimum height=2em]
    
\usetikzlibrary{positioning}
\tikzset{main node/.style={circle,fill=blue!20,draw,minimum size=1cm,inner sep=0pt},  }

\newcommand{\argmin}{\text{arg}\min}
\newcommand{\argmax}{\text{arg}\max}

\begin{document}
\title[]{Controlling conservation laws II: compressible Navier--Stokes equations} 
\author[Li]{Wuchen Li}
\email{wuchen@mailbox.sc.edu}
\address{Department of Mathematics, University of South Carolina, Columbia}
\author[Liu]{Siting Liu}
\email{siting6@ucla.edu}
\address{Department of Mathematics, University of California, Los Angeles}

\author[Osher]{Stanley Osher}
\email{sjo@math.ucla.edu}
\address{Department of Mathematics, University of California, Los Angeles}
\newcommand{\vr}{\overrightarrow}
\newcommand{\wt}{\widetilde}
\newcommand{\dd}{\mathcal{\dagger}}
\newcommand{\ts}{\mathsf{T}}
\thanks{W. Li thanks the start-up funding from the University of South Carolina and NSF RTG: 2038080. In addition, W. Li, S. Liu and S. Osher thank the funding from AFOSR MURI FA9550-18-1-0502 and ONR grants: N00014-18-1-2527, N00014-20-1-2093, and N00014-20-1-2787.}
\begin{abstract}
We propose, study, and compute solutions to a class of optimal control  problems for hyperbolic systems of conservation laws and their viscous regularization \cite{Lax}. We take barotropic compressible Navier--Stokes equations (BNS) as a canonical example. We first apply the entropy--entropy flux--metric condition for BNS. We select an entropy function and rewrite BNS to a summation of flux and metric gradient of entropy. We then develop a metric variational problem for BNS, whose critical points form a primal-dual BNS system. We design a finite difference scheme for the variational system. The numerical approximations of conservation laws are implicit in time. We solve the variational problem with an algorithm inspired by the primal--dual hybrid gradient method. This includes a new method for solving implicit time approximations for conservation laws, which seems to be unconditionally stable. {\color{black}Several numerical examples are presented to demonstrate the effectiveness of the proposed algorithm.} 
\end{abstract}
\maketitle
\section{Introduction}
Nonlinear systems of conservation laws \cite{Evans,LF} play essential roles in physics, modeling, engineering, and scientific computing with potential applications in AI (Artificial intelligence) and Bayesian sampling problems. A canonical example of systems of conservation laws is the compressible Navier--Stokes equations \cite{BDKL}. They describe the fluid flow using physical laws, such as conservation of mass, momentum and energy. The system also contains a viscosity term, which describes thermodynamics' dissipative nature. Solving compressible Navier--Stokes equations and their simplifications are fundamental problems in computational fluid dynamics.   

In this paper, we propose a class of optimal control problems for systems of conservation laws following \cite{LLO}. We select the barotropic compressible Navier--Stokes equation (BNS) as an example. We first apply the entropy--entropy flux--metric condition for BNS. We then select an entropy function and rewrite BNS into the summation of flux and metric gradient of entropy. We call this formulation ``flux-gradient flow'' in BNS metric space. We use the flux-gradient flow formulation to design a metric variation problem and derive its critical point system, i.e., the primal--dual BNS system. We demonstrate that the primal-dual BNS system is useful in modeling and computation. {\color{black} More importantly, we apply a primal-dual hybrid gradient method and Lax–Friedrichs type schemes to compute the primal--dual BNS system. It includes a simple-to-implement  method for solving implicit time approximations for conservation laws, which seem to be unconditionally stable. We present several numerical examples to demonstrate the effectiveness of the method.}

The main result is sketched below. Denote $\Omega$ as a one dimensional torus, and define $\mathcal{F}$, $\mathcal{G}$ as smooth functionals. Consider a variational problem for BNS: 
\begin{equation*}
\inf_{\rho, m, a, \rho_1, m_1}~\int_0^1\Big[\int_\Omega \frac{1}{2}|a(t,x)|^2\mu(\rho(t,x)) dx-\mathcal{F}(\rho, m)(t)\Big]dt+\mathcal{H}(\rho_1,m_1),
\end{equation*}
where the infimum is taken among variables $\rho\colon [0,1]\times \Omega\rightarrow\mathbb{R}_+$, $m\colon [0,1]\times\Omega\rightarrow\mathbb{R}$, $a\colon [0,1]\times \Omega\rightarrow\mathbb{R}$, and $\rho_1\colon\Omega\rightarrow\mathbb{R}_+$, $m_1\colon\Omega\rightarrow\mathbb{R}$ satisfying 
\begin{equation*}
\left\{\begin{aligned}
&\partial_t\rho+\partial_x m=0,\\
&\partial_t m+\partial_x( \frac{m^2}{\rho})+\partial_xP(\rho)+\partial_x(\mu(\rho)a)=\beta\partial_x(\mu(\rho)\partial_x\frac{m}{\rho}),
\end{aligned}\right.
\end{equation*}
with given initial time value conditions $\rho(0,x)=\rho_0(x)$, $m(0,x)=m_0(x)$. Here we assume $P(\rho)=\rho^\gamma$, $\mu(\rho)=\rho^\alpha$, $\gamma, \alpha \in\mathbb{R}$. The critical point system of the above variational problem is described below. Denote $\phi$, $\psi\colon [0, 1]\times \Omega\rightarrow\mathbb{R}$. Then $a(t,x)=\partial_x\psi(t,x)$, and
\begin{equation*}\left\{\begin{aligned}
&\partial_t\rho+\partial_x m=0,\\
&\partial_t m+\partial_x( \frac{m^2}{\rho})+\partial_xP(\rho)+\partial_x(\mu(\rho)\partial_x\psi)=\beta\partial_x(\mu(\rho)\partial_x\frac{m}{\rho}),\\
&\partial_t\phi+\frac{1}{2}|\partial_x\psi|^2\mu'(\rho)-(\frac{m^2}{\rho^2}, \partial_x\psi)+(P'(\rho), \partial_x\psi)+\frac{\delta}{\delta\rho}\mathcal{F}(\rho, m)\\
&\hspace{3cm}=\beta(\partial_x\psi,\partial_x\frac{m}{\rho})\mu'(\rho)+\beta\frac{m}{\rho^2}\partial_x(\mu(\rho)\partial_x\psi),\\
&\partial_t\psi+2\partial_x\psi\cdot\frac{m}{\rho}+\partial_x\phi+\frac{\delta}{\delta m}\mathcal{F}(\rho, m)=-\beta\frac{1}{\rho}\partial_x(\mu(\rho)\partial_x\psi).
\end{aligned}\right.
\end{equation*}
Here functions $\phi$, $\psi$ have boundary conditions at the terminal time $t=1$. We call the above system the primal-dual BNS system. Clearly, if we select $\mathcal{F}=\mathcal{H}=0$, then we minimize a quadratic running cost in term of $a^2$, in which $a=0$ is a critical point solution. The primal--dual BNS system forms the initial value problem of BNS equation.

In the literature, optimal control problems in density space are widely considered in optimal transport \cite{AGS, BB,C1,O1,otto2001,Villani2009_optimal}, mean--field games \cite{MFGC,MRPP,LL},  
and Schr{\"o}dinger bridge problems \cite{BCGL,CGP,LeL}. These control problem are often studied on a scalar density function. We extend current studies in modeling systems of conservation laws,  where we study the dynamics of the density and its momentum as a system. We also remark that the entropy--entropy flux--metric condition is closely related to the energetic variational approach in the literature \cite{GZW, CL,CLW}. In this paper, we choose both entropy (Lyapunov) functionals and optimal transport type metrics from the flux function. Under this selection, we design a class of optimal control problems for systems of conservation laws, from which we derive primal--dual systems of conservation laws and design implicit variational schemes.

The paper is organized as follows. In section \ref{section2}, we briefly review the conservation laws with entropy--entropy flux--conditions. We further design control problems for flux--gradient flows. In section \ref{section3}, we apply this approach to control barotropic compressible Navier--Stokes equations and derive their primal-dual PDE systems. {\color{black}In section \ref{section4}, we formulate primal-dual hybrid gradient like algorithms to solve the BNS system numerically.  Several numerical examples are presented.}
\section{Conservation law and entropy-entropy flux-metric}\label{section2}
In this section, we present the entropy--entropy flux--metric condition for regularized systems of conservation laws \cite{LLO}. Following this condition, we define a class of metric operators for systems of conservation laws, and then design flux-mean-field control problems. 

\subsection{Entropy--entropy flux--metric}
For simplicity of presentation, we consider a one dimensional periodic spatial domain. I.e., $\Omega=\mathbb{T}^1$.  
Consider a system of $N$ partial differential equations  
\begin{equation}\label{cld}
\partial_t u_i(t,x)+\partial_x f_i(u(t,x))=\beta \sum_{j=1}^N\partial_x(A_{ij}(u(t,x))\partial_xu_j(t,x)),
\end{equation}
where $u=(u_1,\cdots, u_N)$ is a vector function with $u_i\colon \mathbb{R}_+\times\Omega\rightarrow\mathbb{R}^1$, $i=1,\cdots, N$, 
$f=(f_1, \cdots f_N)$ is a flux vector function with $f_i\colon \mathbb{R}^N\rightarrow\mathbb{R}^1$, $i=1\cdots, N$, and $A=(A_{ij})_{1\leq i,j\leq N}\in\mathbb{R}^{N\times N}$ is a semi-positive definite matrix function with $A_{ij}\colon \mathbb{R}^{N}\rightarrow \mathbb{R}^1$, $i,j=1,\cdots, N$.

We next define a metric space for the unknown vector function $u$. Here the metric is constructed by both entropy-entropy flux condition and the nonlinear diffusion operator.  
\begin{definition}[Entropy--entropy flux--metric condition]\label{def2}
We call $(G,\Psi, C)$ an entropy-entropy flu-metric condition for equation \eqref{cld} if there exists a convex function $G\colon \mathbb{R}^N\rightarrow\mathbb{R}$, and $\Psi\colon \mathbb{R}^N\rightarrow\mathbb{R}$, such that
\begin{equation*}
\frac{\partial}{\partial u_i}\Psi(u)=\sum_{j=1}^N\frac{\partial}{\partial u_j}G(u) \frac{\partial}{\partial u_i}f_j(u),
\end{equation*}
and there exists a symmetric semi-positive matrix function $C\colon \mathbb{R}^N\rightarrow\mathbb{R}^{N\times N}$, such that 
\begin{equation*}
C(u)\nabla^2_{uu}G(u)=A(u).
\end{equation*}
In other words, denote $C=(C_{ij})_{1\leq i,j\leq N}$, such that 
\begin{equation*}
\sum_{j=1}^NC_{ij}(u)\frac{\partial^2}{\partial u_j\partial u_k}G(u)=A_{ik}(u). 
\end{equation*}
We require that $C_{ij}=C_{ji}$ and $C\succeq 0$. Here we call $G$ the {\em entropy element}, $\Psi$ the {\em entropy flux} and $C$ the {\em metric element}. 
\end{definition}
\begin{remark}[Symmetry conditions]
The entropy--entropy flux--metric condition is to require the following symmetric conditions on the regularized conservation law \eqref{cld}. Assume that $G$ is strictly convex. For any $i,k=1,\cdots, N$, 
\begin{itemize}
\item[(i)]
\begin{equation*}
\sum_{j=1}^N\frac{\partial^2}{\partial u_j\partial u_k}G(u) \frac{\partial}{\partial u_i}f_j(u)=\sum_{j=1}^N\frac{\partial^2}{\partial u_j\partial u_i}G(u) \frac{\partial}{\partial u_k}f_j(u);
\end{equation*}
\item[(ii)] 
\begin{equation*}
\Big(A(u)(\nabla^2_{uu}G(u))^{-1}\Big)_{ik}=\Big(A(u)(\nabla^2_{uu}G(u))^{-1}\Big)_{ki},
\end{equation*}
and 
\begin{equation*}
A(u)(\nabla^2_{uu}G(u))^{-1}\succeq 0. 
\end{equation*}
\end{itemize}
We comment that condition (i) follows from the fact that $\frac{\partial^2}{\partial u_i\partial u_k}\Psi(u)=\frac{\partial^2}{\partial u_k\partial u_i}\Psi(u)$, as discussed in Friedrichs-Lax's paper \cite{LF}. Condition (ii) guarantees the existence of generalized optimal transport type metric and generalized Fisher information functional.
\end{remark}

\subsection{Metrics and flux--gradient flows}
From the entropy-entropy flux--metric condition, we introduce the metric space for variable $u$.  
Define the space of functions $u$ as
 \begin{equation*}
 \mathcal{M}=\Big\{u=(u_1,\cdots, u_N)\in C^{\infty}(\Omega)^N\colon \int_\Omega u_i(x)dx=\mathrm{constant},\quad \textrm{for $i=1,\cdots,N$}\Big\}.
 \end{equation*}
Denote the tangent space of $\mathcal{M}(u)$ at point $u$ as
\begin{equation*}
T_u\mathcal{M}=\Big\{\sigma=(\sigma_1,\cdots, \sigma_N)\in C^{\infty}(\Omega)^N\colon \int_\Omega \sigma_i(x)dx=0, \quad \textrm{for $i=1,\cdots,N$}\Big\}.
\end{equation*}
We define a metric operator on the vector function space $\mathcal{M}$. Here we shall use the metric element $C(u)$. 
\begin{definition}[Metric]\label{metrics}
  Define the inner product $\g\colon\mathcal{M}\times
  {T_u}\mathcal{M}\times{T_u}\mathcal{M}\rightarrow\mathbb{R}$ below. 
 \begin{equation*}\begin{split}
    \g(u)(\sigma, \hat\sigma)  =&\sum_{i,j=1}^N\int_\Omega (\partial_x\phi_i(x), \partial_x\tilde\phi_j(x))C_{ij}(u)dx,
  \end{split}
\end{equation*}
where vector functions $\phi=(\phi_1,\cdots,\phi_N)$, $\tilde\phi=(\tilde\phi_1,\cdots,\tilde\phi_N)\in C^{\infty}(\Omega)^N$ satisfy
\begin{equation*}
\sigma_i=-\sum_{j=1}^N\partial_x(C_{ij}(u)\partial_x\phi_j), \qquad \tilde\sigma_i=-\sum_{j=1}^N\partial_x(C_{ij}(u)\partial_x\tilde\phi_j), 
\end{equation*}
for $i=1,\cdots, N$.
\end{definition}

In this metric space $(\mathcal{M}, \g)$, we notice that the dissipative operator of PDE \eqref{cld} forms the gradient descent flow of the entropy functional. We denote the entropy functional as
\begin{equation*}
\mathcal{G}(u)=\int_\Omega G(u)dx.
\end{equation*}
\begin{proposition}[Gradient flow]\label{gd}
The gradient descent flow of functional $\mathcal{G}(u)$ in $(\mathcal{M}, \g)$ satisfies
\begin{equation*}
\begin{split}
\partial_t u_i=&\sum_{j=1}^N\partial_x\Big(C_{ij}(u)\partial_x\frac{\partial}{\partial u_j}G(u)\Big)
=\sum_{j=1}^N\partial_x(A_{ij}(u)\partial_xu_j).
\end{split}
\end{equation*}
\end{proposition}
\begin{proof}
The proof is based on a direct computation. 
\begin{equation*}
\begin{split}
\partial_t u_i=&\sum_{j=1}^N\partial_x\Big(C_{ij}(u)\partial_x\frac{\partial}{\partial u_j}G(u)\Big)\\
=&\sum_{j=1}^N\sum_{k=1}^N\partial_x\Big(C_{ij}(u)\frac{\partial^2}{\partial u_j\partial u_k}G(u)\partial_xu_k\Big)\\
=&\sum_{k=1}^N\partial_x\Big(A_{ik}(u)\partial_xu_k\Big).
\end{split}
\end{equation*}
In the second equality, we use the fact that $\sum_{j=1}^NC_{ij}(u)\frac{\partial^2}{\partial u_j\partial u_k}G(u)=A_{ik}(u)$. 
\end{proof}
Under the metric space, the conservation law system \eqref{cld} has a ``flux--gradient flow'' formulation. The {\em flux--gradient flows} demonstrate the dissipation behavior of regularized systems of conservation laws with entropy-entropy flux pairs.
\begin{definition}[Flux--gradient flow]\label{def5}
Equation \eqref{cld} can be written as  
\begin{equation*}
\partial_t u_i+\partial_x f_i(u)=\beta\sum_{j=1}^N\partial_x\Big(C_{ij}(u)\partial_x\frac{\delta}{\delta u_j}\mathcal{G}(u)\Big),
\end{equation*}
where 
\begin{equation*}
\sum_{i=1}^N\int_\Omega f_i(u)\cdot\partial_x\frac{\delta}{\delta u_i(x)}\mathcal{G}(u)dx=0. 
\end{equation*}
We denote the above formulation of equation \eqref{cld} as the {\em flux--gradient flow} in $(\mathcal{M}, \g)$. 
\end{definition}
\begin{corollary}[Entropy--Entropy flux--Fisher information dissipation]
Energy functional $\mathcal{G}(u)$ is a Lyapunov functional for PDE \eqref{cld}. Suppose $u(t,x)$ is the solution of equation \eqref{cld}, then 
\begin{equation*}
\frac{d}{dt}\mathcal{G}(u(t,\cdot))=-\beta\mathcal{I}_{\mathcal{G}}(u(t,\cdot))\leq 0,
\end{equation*}
where $\mathcal{I}_{\mathcal{G}}\colon \mathcal{M}\rightarrow\mathbb{R}_+$ is the ``generalized Fisher information functional'' defined as
\begin{equation*}
\begin{split}
\mathcal{I}_{\mathcal{G}}(u)
=&\sum_{i,j=1}^N\int_\Omega \partial_x\frac{\partial}{\partial u_i}G(u)\cdot \partial_x\frac{\partial}{\partial u_j}G(u)\cdot C_{ij}(u(x))dx. 
\end{split}
\end{equation*}
\end{corollary}
\begin{proof}
The proof follows from the entropy-entropy flux-metric condition and integration by parts. In detail,  
\begin{equation*}
\begin{split}
\frac{d}{dt}\mathcal{G}(u(t,\cdot))
=&\sum_{i=1}^N\int_\Omega \frac{\partial}{\partial u_i}G(u)\partial_t u_i dx\\
=&-\sum_{i=1}^N\int_\Omega \frac{\partial}{\partial u_i}G(u) \partial_xf_i(u)dx+\beta\sum_{i,j=1}^N\int_\Omega \frac{\partial}{\partial u_i}G(u) \partial_x\Big(C_{ij}(u)\partial_x\frac{\partial}{\partial u_j}G(u)\Big)\\
=&-\sum_{i,j=1}^N\int_\Omega\partial_x\frac{\partial}{\partial u_i}G(u) \frac{\partial}{\partial u_j}f_i(u)\partial_x u_jdx-\beta\sum_{i,j=1}^N\int_\Omega C_{ij}(u)\partial_x\frac{\partial}{\partial u_i}G(u) \partial_x\frac{\partial}{\partial u_j}G(u)dx\\
=&-\sum_{j=1}^N\int_\Omega \frac{\partial}{\partial u_j}\Psi(u)\partial_x u_jdx-\beta\sum_{i,j=1}^N\int_\Omega C_{ij}(u)\partial_x\frac{\partial}{\partial u_i}G(u) \partial_x\frac{\partial}{\partial u_j}G(u)dx\\
=&-\sum_{j=1}^N\int_\Omega \partial_x\Psi(u)dx-\beta\sum_{i,j=1}^N\int_\Omega C_{ij}(u)\partial_x\frac{\partial}{\partial u_i}G(u) \partial_x\frac{\partial}{\partial u_j}G(u)dx\\
=&-\beta\sum_{i,j=1}^N\int_\Omega C_{ij}(u)\partial_x\frac{\partial}{\partial u_i}G(u) \partial_x\frac{\partial}{\partial u_j}G(u)dx. 
\end{split}
\end{equation*}
\end{proof}
\begin{remark}
In the literature, the dissipation of entropy along diffusion equals to the negative Fisher information functional. I.e., $N=1$, $G(u)=u\log u-u$, $f=0$, $C(u)=u$. Then 
\begin{equation*}
    \partial_t\int_\Omega G(u)dx=-\int_\Omega |\partial_x\log u|^2u dx. 
    \end{equation*}
The above fact follows directly from the gradient flow formalism in optimal transport metric \cite{otto2001}. Indeed, the similar dissipation relation also holds for flux--gradient flows in a general metric space $(\mathcal{M}, \g)$. We call the functional $\mathcal{I}_{\mathcal{G}}$ ``generalized Fisher information functional''. In next section, we derive the barotropic Navier--Stokes metric and its Fisher information functional. 
\end{remark}

\subsection{Controlling flux--gradient flows}
In this subsection, we construct the optimal control problems for flux-gradient flows. This is to design an optimal control problem over flux--gradient flows in a metric space. 
\begin{definition}[Optimal control of conservation laws]
\begin{subequations}\label{MFC}
Given smooth functionals $\mathcal{F}$, $\mathcal{H}\colon \mathcal{M}\rightarrow\mathbb{R}$, consider a variational problem
\begin{equation}\label{MFC1}
\inf_{u, v, u_1}~\int_0^1\Big[\frac{1}{2}\int_\Omega \sum_{i,j=1}^NC_{ij}(u)v_iv_jdx-\mathcal{F}(u)\Big]dt+\mathcal{H}(u_1),
\end{equation}
where the infimum is taken among variables $v\colon [0,1]\times \Omega\rightarrow\mathbb{R}^N$, $u\colon [0,1]\times\Omega\rightarrow\mathbb{R}^N$, and $u_1\colon\Omega\rightarrow\mathbb{R}^N$ satisfying 
\begin{equation}\label{MFC2}
\begin{split}
\partial_t u_i+\partial_x f_i(u)+\sum_{j=1}^N\partial_x(C_{ij}(u)v_j)=\beta \sum_{j=1}^N\partial_x(A_{ij}(u)\partial_xu_j),\quad u(0,x)=u^0(x).
\end{split}
\end{equation}
Here $u^0\colon\Omega\rightarrow \mathbb{R}^N$ is a fixed initial value vector function. 
\end{subequations}
\end{definition}

We next derive critical point systems of variational problem \eqref{MFC}. They are Hamiltonian flows in $(\mathcal{M}, \g)$ associated with regularized conservation laws.   
\begin{proposition}[Hamiltonian flows of conservation laws]\label{MFH}
A critical point system of variational problem \eqref{MFC} is given below. There exists a vector function $\phi\colon [0,1]\times\Omega\rightarrow\mathbb{R}^N$, such that 
\begin{equation*}
v_i(t,x)=\partial_x\phi_i(t,x),
\end{equation*}
and
\begin{equation}\label{MFCeq}
\left\{\begin{aligned}
&\partial_t u_i+\partial_x f_i(u)+\sum_{j=1}^N\partial_x(C_{ij}(u)\partial_x\phi_j)=\beta \sum_{j=1}^N\partial_x(A_{ij}(u)\partial_xu_j),\\
&\partial_t\phi_i+\sum_{k=1}^N\partial_x\phi_k \frac{\partial}{\partial u_i}f_k(u)+\frac{1}{2}\sum_{j,k=1}^N\partial_x\phi_j\partial_x\phi_k\frac{\partial}{\partial u_i}C_{jk}(u)+\frac{\delta}{\delta u_i}\mathcal{F}(u)\\
&\hspace{2cm}=-\beta \sum_{j=1}^N\partial_x(A_{ji}(u)\partial_x\phi_j)+\beta\sum_{j,k=1}^N\partial_x\phi_j\partial_xu_k\frac{\partial}{\partial u_i}A_{jk}(u).
\end{aligned}\right.
\end{equation}
Here initial and terminal time conditions satisfy
\begin{equation*}
u_i(0,x)=u_i^0(x), \quad \frac{\delta}{\delta u_i^1}\mathcal{H}(u^1)+\phi_i(1,x)=0,\quad i=1,\cdots, N. 
\end{equation*} 
\end{proposition}
\begin{proof}
Denote a Lagrange multiplier vector function $\phi=(\phi_1,\cdots,\phi_N)$. Consider the following saddle point problem 
\begin{equation*}
\inf_{u,v,u_1}\sup_\phi~\mathcal{L}(u,v,u_1,\phi),
\end{equation*}
where
\begin{equation*}
\begin{split}
\mathcal{L}(u,v,u_1,\phi)=&\int_0^1\Big[\frac{1}{2}\int_\Omega \sum_{i,j=1}^NC_{ij}(u)v_iv_jdx-\mathcal{F}(u)\Big]dt+\mathcal{H}(u_1)\\
&+\int_0^1\int_\Omega\sum_{i=1}^N\phi_i\Big(\partial_t u_i+\partial_x f_i(u)+\sum_{j=1}^N\partial_x(C_{ij}(u)\partial_x\phi_j)-\beta \sum_{j=1}^N\partial_x(A_{ij}(u)\partial_xu_j)\Big)dxdt.
\end{split}
\end{equation*}
 The saddle point system satisfies
 \begin{equation*}
 \left\{\begin{aligned}
&\frac{\delta}{\delta v_i}\mathcal{L}=0,\\
&\frac{\delta}{\delta \phi_i}\mathcal{L}=0,\\
&\frac{\delta}{\delta u_i}\mathcal{L}=0,\\
&\frac{\delta}{\delta u^1_i}\mathcal{L}=0.
\end{aligned}\right.
\end{equation*}
In detail, we have
\begin{equation*}
\left\{\begin{aligned}
&\sum_{j=1}^NC_{ij}(u)(v_i-\partial_x\phi_i)=0,\\
&\partial_t u_i+\partial_x f_i(u)+\sum_{j=1}^N\partial_x(C_{ij}(u)\partial_x\phi_j)-\beta \sum_{j=1}^N\partial_x(A_{ij}(u)\partial_xu_j)=0,\\
&\frac{1}{2}\sum_{k,l=1}^N\frac{\partial}{\partial u_i}C_{kl}(u)v_kv_l-\frac{\delta}{\delta u_i}\mathcal{F}(u)-\partial_t\phi_i-\sum_{k=1}^K\partial_x\phi_k\frac{\partial}{\partial u_i}f_k(u)\\
&-\sum_{k,l=1}^N\frac{\partial}{\partial u_i}C_{kl}(u)\partial_x\phi_k\partial_x\phi_l-\beta \sum_{j=1}^N\partial_x(A_{ji}(u)\partial_x\phi_j)+\beta\sum_{j,k=1}^N\partial_x\phi_j\partial_xu_k\frac{\partial}{\partial u_i}A_{jk}(u)=0,\\
&\frac{\delta}{\delta u_i^1}\mathcal{H}(u^1)+\phi_i(1,x)=0.
\end{aligned}\right.
 \end{equation*} 
By substituting $v_i=\partial_x\phi_i$ into the third equality, we finish the derivation.  
\end{proof}
\begin{proposition}
PDE system \eqref{MFCeq} has the following Hamiltonian flow formulation in $(\mathcal{M}, \g)$. For $i=1,\cdots, N$, 
\begin{equation*}
\left\{\begin{aligned}
\partial_tu_i=&\frac{\delta}{\delta\phi_i}\mathcal{H}_{\mathcal{G}}(u, \phi),\\
 \partial_t\phi_i=&-\frac{\delta}{\delta u_i}\mathcal{H}_{\mathcal{G}}(u, \phi),
\end{aligned}\right.
\end{equation*}
where we define a Hamiltonian functional $\mathcal{H}_{\mathcal{G}}\colon \mathcal{M}\times C^{\infty}(\Omega)^N\rightarrow\mathbb{R}$ as
\begin{equation}\label{eq:hamiltonian_functional}
\mathcal{H}_{\mathcal{G}}(u,\phi)=\int_\Omega \sum_{i,j=1}^N\Big[\frac{1}{2}C_{ij}(u)\partial_x\phi_i\partial_x\phi_j-\beta A_{ij}(u)\partial_x\phi_i\partial_xu_j\Big]dx+\int_\Omega\sum_{k=1}^N\Big[\partial_x\phi_k f_k(u)\Big] dx+\mathcal{F}(u).
\end{equation}
In addition, the Hamilton-Jacobi equation in $(\mathcal{M}, \g)$ satisfies
\begin{equation*}
\begin{split}
&\partial_t\mathcal{U}(t,u)+\frac{1}{2}\sum_{i,j=1}^N\int_\Omega\partial_x\frac{\delta}{\delta u_i(x)}\mathcal{U}(t,u)\cdot \partial_x\frac{\delta}{\delta u_j(x)}\mathcal{U}(t,u)\cdot C_{ij}(u(x))dx\\
&\hspace{1.4cm}+\sum_{k=1}^N\int_\Omega\partial_x\frac{\delta}{\delta u_k(x)}\mathcal{U}(t,u)\cdot f_k(u(x))dx+\mathcal{F}(u)\\
&\hspace{1.4cm}-\beta \sum_{i,j=1}^N\int_\Omega \partial_x\frac{\delta}{\delta u_i(x)}\mathcal{U}(t,u)\cdot\partial_x u_j(x)\cdot C_{ij}(u(x))dx=0,
\end{split}
\end{equation*}
where $\mathcal{U}\colon [0,1]\times L^2(\Omega)^N\rightarrow\mathbb{R}$ is a value functional. 
\end{proposition}
\begin{proof}
The proof follows from a direct calculation. See detailed derivations in \cite{LLO}.  
\end{proof}
\section{Controlling barotropic compressible Navier--Stokes equations}\label{section3}
In this section, we present an example for control problems of systems of conservation laws. 

We study one dimensional barotropic compressible Navier--Stokes equations. We shall derive a primal-dual system for this system. Consider 
\begin{equation}\label{BNS}
\left\{\begin{aligned}
&\partial_t\rho+\partial_x(\rho v)=0,\\
&\partial_t (\rho v)+\partial_x(\rho v^2)+\partial_xP(\rho)=\beta \partial_x(\mu(\rho)\partial_xv).
\end{aligned}\right.
\end{equation}
Here $\rho=\rho(t,x)$ is the density function, $v=v(t,x)$ is the vector-valued velocity function and $\beta>0$ is diffusion constant. For simplicity, let $\rho$ stay in one dimensional compact spatial domain with periodic boundary conditions. E.g., $\Omega=\mathbb{T}^1$. And the pressure term $P(\rho)$ and the viscosity coefficient $\mu(\rho)$ are smooth functions of variable $\rho$. E.g.,
 \begin{equation*}
P(\rho)=\rho^\gamma, \qquad\mu(\rho)=\rho^\alpha,
\end{equation*}
where $\gamma>1$ and $\alpha\in \mathbb{R}$ are given constants. The PDE system \eqref{BNS} has a conservation law system formulation. Denote $m=\rho v$, i.e., $v=\frac{m}{\rho}$ when $\rho>0$. In this notation, equation system \eqref{BNS} satisfies 
\begin{equation}\label{BNS1}
\left\{\begin{aligned}
&\partial_t\rho+\partial_x m=0,\\
&\partial_t m+\partial_x( \frac{m^2}{\rho})+\partial_x P(\rho)=\beta\partial_x(\mu(\rho)\partial_x\frac{m}{\rho}).
\end{aligned}\right.
\end{equation}
The system \eqref{BNS1} satisfies 
\begin{equation*}
u=\begin{pmatrix}
\rho\\ m
\end{pmatrix}
\in\mathbb{R}_+\times\mathbb{R}, \quad f(u)=
\begin{pmatrix}
m \\\frac{|m|^2}{\rho}+P(\rho)\end{pmatrix}\in\mathbb{R}^2,\quad \mathcal{C}(u)=\begin{pmatrix}
0 \\
\partial_x(\mu(\rho)\partial_x\frac{m}{\rho})
\end{pmatrix}\in\mathbb{R}^2.
\end{equation*}

\subsection{Entropy--entropy flux--Fisher information dissipation}
In this subsection, we show that system \eqref{BNS} satisfies the entropy-entropy flux--metric--Fisher information conditions. 
\begin{proposition}[Entropy-entropy flux-metric-Fisher information]
There exists an entropy function, entropy flux, Fisher information and metric operator for equation \eqref{BNS1}.
\begin{itemize}
\item[(i)]{\em Entropy-entropy flux}:
Denote an entropy function $G\colon \mathbb{R}_+\times\mathbb{R}\rightarrow\mathbb{R}$ and an entropy flux $\Psi\colon\mathbb{R}_+\times\mathbb{R}\rightarrow\mathbb{R}$, such that 
\begin{equation*}
G(\rho, m)=\frac{m^2}{2\rho}+\hat P(\rho),\quad \Psi(\rho, m)=\frac{m^3}{2\rho^2}+\hat P'(\rho)m,
\end{equation*}
where $\hat P\colon \mathbb{R}_+\rightarrow\mathbb{R}$ is a function satisfying  
\begin{equation*}
\hat P''(\rho)=\frac{P'(\rho)}{\rho}.
\end{equation*}
Suppose $\left(\rho(t,x), m(t,x)\right)$ satisfies equation \eqref{BNS} with $\beta=0$. Then the following entropy solution condition hold. 
\begin{equation*}
\partial_t G(\rho(t,x), m(t,x))+\partial_x\Big(\Psi(\rho(t,x), m(t,x))\Big)\leq 0. 
\end{equation*}
\item[(ii)]{\em Metric}: Consider a space 
\begin{equation*}
\mathcal{M}=\Big\{(\rho, m)\in C^{\infty}(\Omega)^2\colon \rho>0,~~\int_\Omega \rho dx=c_1,~~\int_\Omega m dx=c_2,~~\textrm{where $c_1>0$, $c_2\in\mathbb{R}$}\Big\}. 
\end{equation*} 
The tangent space of $\mathcal{M}$ at $(\rho, m)$ satisfies 
\begin{equation*}
T_u\mathcal{M}=\Big\{(\dot\rho, \dot m)\in C^{\infty}(\Omega)\times C^{\infty}(\Omega)\colon \int_\Omega \dot \rho dx=0,\quad \int_\Omega \dot m dx=0\Big\}.
\end{equation*}
In this case, the (degenerate) metric $\g\colon \mathcal{M}\times T_u\mathcal{M}\times T_u\mathcal{M}\rightarrow\mathbb{R}$ satisfies
\begin{equation*}
\g(\rho, m)((\dot\rho_1, \dot m_1), (\dot\rho_2, \dot m_2))=\int_\Omega \partial_x\psi_1(x)\cdot\partial_x\psi_2(x)\cdot\mu(\rho(x)) dx,
\end{equation*}
where $(\dot\rho_i, \dot m_i)\in T_u\mathcal{M}$ and $(\dot m_i, \psi_i)$ satisfies the following parabolic equation 
\begin{equation*}
\dot m_i=-\partial_x(\mu(\rho)\partial_x\psi_i),\quad i=1,2. 
\end{equation*}
\item[(iii)] {\em Fisher information dissipation}: Denote an entropy functional $\mathcal{G}\colon \mathcal{M}\rightarrow\mathbb{R}$ as 
\begin{equation*}
\mathcal{G}(\rho, m)=\int_\Omega G(\rho(x), m(x))dx. 
\end{equation*}
Suppose $(\rho(t,x), m(t,x))$ satisfies equation system \eqref{BNS}, then $\mathcal{G}$ is a Lyapunov functional. In detail, the following dissipation holds. 
\begin{equation*}
\begin{split}
\frac{d}{dt}\mathcal{G}(\rho(t,\cdot), m(t,\cdot))=&-\beta\mathcal{I}_{\mathcal{G}}(\rho(t,\cdot), m(t,\cdot))\leq 0, 
\end{split}
\end{equation*}
where $\mathcal{I}_{\mathcal{G}}\colon\mathcal{M}\rightarrow\mathbb{R}_+$ is a Fisher information functional defined as 
\begin{equation*}
\begin{split}
\mathcal{I}_{\mathcal{G}}(\rho, m)=&\int_\Omega |\partial_x\frac{\delta}{\delta m}\mathcal{G}(\rho(x), m(x))|^2\mu(\rho(x))dx\\
=&\int_\Omega |\partial_x\frac{m(x)}{\rho(x)}|^2 \mu(\rho(x))dx. 
\end{split}
\end{equation*}
\end{itemize}
\end{proposition}
\begin{proof}
(i) We first apply Lax's entropy-entropy flux condition \cite{Evans, Lax}. We need to find both entropy and entropy flux function. Denote $(\rho, m)$ as a solution for dynamics \eqref{CBNS} with $\beta=0$. By a direct computation, we have
\begin{equation*}
\begin{split}
\frac{\partial}{\partial t}G(\rho, m)=&G_\rho(\rho, m)\partial_t\rho+G_m(\rho, m)\partial_t m\\
=&-G_\rho(\rho, m)\partial_x m-G_m(\rho, m)\big(\partial_x(\frac{m^2}{\rho})+\partial_xP(\rho)\big)\\
=&-\Big\{G_\rho(\rho, m)\partial_xm+G_m(\rho, m)\frac{2m}{\rho}\partial_xm+G_m(\rho, m)\frac{m^2}{\rho^2}\partial_x\rho-G_m(\rho, m)P'(\rho)\partial_x\rho\Big\}\\
=&-\Big\{G_\rho(\rho, m)+G_m(\rho, m)\frac{2m}{\rho}\Big\}\partial_xm-\Big\{-G_m(\rho, m)\frac{m^2}{\rho^2}+G_m(\rho, m)P'(\rho)\Big\}\partial_x\rho.
\end{split}
\end{equation*}
Clearly, the entropy-entropy flux condition requires that there exists a function $\Psi\colon \mathbb{R}_+\times\mathbb{R}\rightarrow\mathbb{R}$, such that 
\begin{equation*}
\left\{\begin{aligned}
&\Psi_\rho(\rho, m)=G_m(\rho, m)\Big(-\frac{m^2}{\rho^2}+P'(\rho)\Big),\\
&\Psi_m(\rho, m)=G_\rho(\rho, m)+G_m(\rho, m)\frac{2m}{\rho}.
\end{aligned}\right.
\end{equation*}
This is to enforce the condition $\Psi_{\rho m}=\Psi_{m\rho}$. In other words, we need to solve the following PDE:
\begin{equation*}
\Big(-G_m(\rho, m)\frac{m^2}{\rho^2}+G_m(\rho, m)P'(\rho)\Big)_m=\Big(G_\rho(\rho, m)+G_m(\rho, m)\frac{2m}{\rho}\Big)_\rho.
\end{equation*}
I.e., 
\begin{equation*}
-G_{m m}(\rho, m)\frac{m^2}{\rho^2}-G_m(\rho, m)\frac{2m}{\rho^2}+G_{mm}(\rho, m)P'(\rho)=G_{\rho\rho}(\rho, m)+G_{m\rho}(\rho, m)\frac{2m}{\rho}-G_m(\rho, m)\frac{2m}{\rho^2}.
\end{equation*}
I.e., 
\begin{equation}\label{ee}
G_{mm}(\rho, m)(P'(\rho)-\frac{m^2}{\rho^2})=G_{\rho\rho}(\rho, m)+G_{m\rho}(\rho, m)\frac{2m}{\rho}.
\end{equation}
Assume that $G$ has a formulation
\begin{equation*}
G(\rho, m):=\frac{m^2}{2\rho^k}+\hat P(\rho).
\end{equation*}
Then equation \eqref{ee} forms 
\begin{equation*}
\left\{\begin{aligned}
G_\rho(\rho, m)=&-\frac{km^2}{2\rho^{k+1}}+\hat P'(\rho),\quad G_{\rho\rho}(\rho, m)=\frac{k(k+1)m^2}{2\rho^{k+2}}+\hat P''(\rho),\\
G_m(\rho, m)=&\frac{m}{\rho^k},\quad G_{m\rho}(\rho,m)=-\frac{km}{\rho^{k+1}},\quad G_{mm}(\rho,m)=\frac{1}{\rho^k}.
\end{aligned}\right.
\end{equation*}
Hence condition \eqref{ee} satisfies 
\begin{equation*}
\frac{1}{\rho^k}(P'(\rho)-\frac{m^2}{\rho^2})=\frac{k(k+1)m^2}{2\rho^{k+2}}+\hat P''(\rho)-\frac{2km^2}{\rho^{k+2}}.
\end{equation*}
I.e., 
\begin{equation*}
(\frac{k(k+1)}{2}-2k+1)\frac{m^2}{\rho^{k+2}}+\hat P''(\rho)-\frac{P'(\rho)}{\rho^k}=0.
\end{equation*}
In this case, $k=1$ or $2$. Here we are only interested in $k=1$, such that  
\begin{equation*}
G(\rho, m)=\frac{m^2}{2\rho}+\hat P(\rho), \quad \textrm{where $\hat P''(\rho, m)=\frac{P'(\rho)}{\rho}$}.
\end{equation*}

(ii), (iii): When $k=1$, we check that the integration of entropy function $G$, i.e. $\mathcal{G}(\rho, m)=\int_\Omega G(\rho, m)dx$, forms a Lyapunov function for dynamics \eqref{BNS1}. Denote $(\rho, m)$ as a solution for dynamics \eqref{BNS1} with $\beta>0$. Then 
\begin{equation*}
\begin{split}
\frac{d}{dt}\mathcal{G}(\rho(t,\cdot), m(t,\cdot))=&\int_\Omega \frac{\partial}{\partial t}G(\rho, m) dx\\
=&\int_\Omega G_\rho(\rho, m)\partial_t\rho+G_m(\rho, m)\partial_t m dx\\
=&\int_\Omega -\partial_x\Psi(\rho, m) dx+\beta \int_\Omega G_m(\rho, m)\partial_x(\mu(\rho)\partial_x\frac{m}{\rho})dx\\
=&\beta \int_\Omega \frac{m}{\rho}\partial_x(\mu(\rho)\partial_x\frac{m}{\rho})dx\\
=&-\beta \int_\Omega |\partial_x\frac{m}{\rho}|^2\mu(\rho)dx.
\end{split}
\end{equation*}
Following the above dissipation behavior, we can define the metric operator. 
See details in section \ref{metric}. 
\end{proof}
\begin{remark}[Entropy flux and generalized Fisher information functional]
We remark that entropy--entropy flux conditions \cite{Evans} are not unique for equation \eqref{BNS1}. There are many entropy functions. In contrast, the proposed metric condition suggests a particular entropy and Fisher information functional. This follows the relation among dissipative operator, entropy and metric behind equation \eqref{BNS1}. In detail, 
\begin{equation*}
\begin{split}
\frac{d}{dt}\mathcal{G}(\rho(t,\cdot), m(t,\cdot))=&-\beta \g((\partial_t\rho, \partial_tm), (\partial_t\rho, \partial_tm))\\
=&-\beta\mathcal{I}_{\mathcal{G}}(\rho(t,\cdot), m(t,\cdot))\\
=&-\beta\int_\Omega |\partial_x\frac{m(t,x)}{\rho(t,x)}|^2\mu(\rho(t,x))dx\leq 0.
\end{split}
\end{equation*}
In the future, we shall study the Navier--Stokes metric operator and demonstrate its connection with the classical Wasserstein-$2$ metric. 
\end{remark}
\subsection{Barotropic compressible Navier--Stokes transport Metrics}\label{metric}
In this subsection, we study the metric operator $\g$ induced by the compressible Navier--Stokes equation \eqref{BNS}. 
We demonstrate that metric, gradient, flux-gradient and Hamiltonian flow dynamics have several coordinates, 
namely tangent space coordinates, and cotangent space coordinates (Eulerian coordinates in fluid dynamics). 

Consider a function space 
\begin{equation*}
\mathcal{M}=\Big\{(\rho, m)\in C^{\infty}(\Omega)^2\colon \rho>0,~~\int_\Omega \rho dx=c_1,~~\int_\Omega m dx=c_2,~~\textrm{where $c_1>0$, $c_2\in\mathbb{R}$}\Big\}. 
\end{equation*} 
The tangent space of $\mathcal{M}$ at $(\rho, m)$ satisfies 
\begin{equation*}
T_u\mathcal{M}=\Big\{(\dot\rho, \dot m)\in C^{\infty}(\Omega)\times C^{\infty}(\Omega)\colon \int_\Omega \dot \rho(x) dx=0,\quad \int_\Omega \dot m(x)dx=0\Big\}.
\end{equation*}
Denote a weighted elliptic operator $\Delta_{\mu(\rho)}\colon C^{\infty}(\Omega)\rightarrow C^{\infty}(\Omega)$ as
\begin{equation*}
\Delta_{\mu(\rho)}=\partial_x(\mu(\rho)\partial_x). 
\end{equation*}
In other words, for any test function $f\in C^{\infty}(\Omega)$, we have 
\begin{equation*}
(\Delta_{\mu(\rho)}f)(x)=\partial_x\Big(\mu(\rho)\partial_xf(x)\Big). 
\end{equation*}
\begin{proposition}[Degenerate $H^{-1}(\rho)$ metric]
Denote $\g\colon \mathcal{M}\times T_u\mathcal{M}\times T_u\mathcal{M}\rightarrow \mathbb{R}$. Then the following formulations of metric operator $\g$ hold.
\begin{itemize}
\item[(i)](Tangent space)
\begin{equation*}
\begin{split}
&\g(\rho, m)((\dot\rho_1, \dot m_1), (\dot\rho_2, \dot m_2))\\=&\int_\Omega \begin{pmatrix}
\dot\rho_1(x)\\
\dot m_1(x)
\end{pmatrix}^{\ts}\begin{pmatrix}
0 & 0\\
0 & (-\Delta_{\mu(\rho)})^{-1}\end{pmatrix}
 \begin{pmatrix}
\dot\rho_2(x)\\
\dot m_2(x)
\end{pmatrix}dx\\
=&\int_\Omega \dot m_1(x) \big((-\Delta_{\mu(\rho)})^{-1}\dot m_2\big)(x) dx. 
\end{split}
\end{equation*}
\item[(ii)](Cotangent space)
\begin{equation*}
\g(\rho, m)((\dot\rho_1, \dot m_1), (\dot\rho_2, \dot m_2))=\int_\Omega (\partial_x\psi_1(x), \partial_x\psi_2(x))\mu(\rho(x)) dx,
\end{equation*}
where $(\dot\rho_i, \dot m_i)\in T_u\mathcal{M}$ and $(\dot m_i, \psi_i)$ satisfies the following parabolic equation 
\begin{equation*}
\dot m_i=-\partial_x(\mu(\rho)\partial_x\psi_i),\quad i=1,2. 
\end{equation*}
\end{itemize}
\end{proposition}
\begin{proposition}[Gradient flows]
Consider a smooth functional $\mathcal{E}\colon \mathcal{M}\rightarrow\mathbb{R}$. The gradient flow of energy functional $\mathcal{E}(\rho, m)$ in $(\mathcal{M}, \g)$ satisfies 
\begin{equation}\label{gd}
\left\{\begin{aligned}
&\partial_t\rho=0,\\
&\partial_t m=\partial_x(\mu(\rho)\partial_x\frac{\delta}{\delta m}\mathcal{E}(\rho, m)).
\end{aligned}\right.
\end{equation}
In particular, if 
\begin{equation*}
\mathcal{E}(\rho,m)=\beta \mathcal{G}(\rho, m)=\beta\Big(\int_\Omega \frac{m^2}{2\rho}dx+\hat P(\rho)\Big),
\end{equation*}
then the gradient flow \eqref{gd} satisfies 
\begin{equation*}
\left\{\begin{aligned}
&\partial_t\rho=0,\\
&\partial_t m=\beta\partial_x(\mu(\rho)\partial_x\frac{m}{\rho}).
\end{aligned}\right.
\end{equation*}
\end{proposition}
\begin{proof}
(i) The gradient flow in $(\mathcal{M}, \g)$ follows by its definition. In other words, 
\begin{equation*}
\begin{split}
\begin{pmatrix}
\partial_t\rho\\ \partial_t m
\end{pmatrix}=&-\begin{pmatrix}
0 & 0\\
0 & -\Delta_{\mu(\rho)}\end{pmatrix}\begin{pmatrix}
\frac{\delta}{\delta\rho}\mathcal{E}(\rho, m)
\\
\frac{\delta}{\delta m}\mathcal{E}(\rho, m)
\end{pmatrix}\\
=&\begin{pmatrix}
0\\
-(-\Delta_{\mu(\rho)})\frac{\delta}{\delta m}\mathcal{E}(\rho, m)
\end{pmatrix}\\
=&\begin{pmatrix}
0\\
\partial_x\big(\mu(\rho)\partial_x\frac{\delta}{\delta m}\mathcal{E}(\rho, m)\big)
\end{pmatrix}.
\end{split}
\end{equation*}
(ii) Since $\mathcal{E}(\rho,m)=\beta \mathcal{G}(\rho, m)=\beta\Big(\int_\Omega \frac{m^2}{2\rho}dx+\hat P(\rho)\Big)$, then 
\begin{equation*}
\frac{\delta}{\delta m}\mathcal{E}(\rho, m)=\beta\frac{m}{\rho}. 
\end{equation*}
Hence the gradient flow \eqref{gd} satisfies
\begin{equation*}
\left\{\begin{aligned}
&\partial_t\rho=0,\\
&\partial_t m=\partial_x(\mu(\rho)\partial_x\frac{\delta}{\delta m}\mathcal{E}(\rho, m))=\beta\partial_x(\mu(\rho)\partial_x\frac{m}{\rho}),
\end{aligned}\right.
\end{equation*}
which finishes the proof. 
\end{proof}
We are now ready to present the flux-gradient flows in $(\mathcal{M}, \g)$. 
\begin{proposition}[Flux-gradient flows]
Consider a smooth functional $\mathcal{E}\colon \mathcal{M}\rightarrow\mathbb{R}$. The flux gradient flow of energy functional $\mathcal{E}(\rho, m)$ in $(\mathcal{M}, \g)$ satisfies 
\begin{equation}\label{fgd}
\left\{\begin{aligned}
&\partial_t\rho+\partial_x f_1(\rho, m)=0,\\
&\partial_t m+\partial_x f_2(\rho, m)=\partial_x(\mu(\rho)\partial_x\frac{\delta}{\delta m}\mathcal{E}(\rho, m)),
\end{aligned}\right.
\end{equation}
where $(f_1, f_2)$ is a flux function assumed to satisfy  
\begin{equation*}
\int_\Omega \Big(f_1(\rho, m)\partial_x\frac{\delta}{\delta \rho}\mathcal{E}(\rho, m)+f_2(\rho, m)\partial_x\frac{\delta}{\delta m}\mathcal{E}(\rho, m)) \Big)dx=0. 
\end{equation*}
In this case, $\mathcal{E}(\rho, m)$ is a Lyapunov functional for equation \eqref{fgd}. In detail, 
\begin{equation*}
\frac{d}{dt}\mathcal{E}(\rho(t,\cdot), m(t,\cdot))=-\int_\Omega |\partial_x\frac{\delta}{\delta m}\mathcal{E}(\rho, m)(t,x)|^2\mu(\rho(t,x))dx. 
\end{equation*}
In particular, if $\mathcal{E}(\rho,m)=\beta\mathcal{G}(\rho, m)=\beta\Big(\int_\Omega \frac{m(x)^2}{2\rho(x)}dx+\hat P(\rho)\Big)$, and $f_1(\rho, m)=m$, $f_2(\rho,m)=\frac{m^2}{\rho}+P(\rho)$, 
then the flux gradient flow \eqref{fgd} forms the barotropic compressible Navier--Stokes equation \eqref{BNS}.   
\end{proposition}

\subsection{Controlling barotropic compressible Navier--Stokes equations}
In this subsection, we present the main result of this paper. We apply the above condition to formulate a variational problem for compressible Navier--Stokes equations. Its critical point system leads to a primal-dual PDE system. 
\begin{definition}[Optimal control of BNS]
\begin{subequations}\label{CNS}
Given smooth functionals $\mathcal{F}$, $\mathcal{H}\colon \mathcal{M}\rightarrow\mathbb{R}$, consider a variational problem
\begin{equation}\label{CNS1}
\inf_{\rho, m, a, \rho_1, m_1}~\int_0^1\Big[\int_\Omega \frac{1}{2}|a(t,x)|^2\mu(\rho(t,x)) dx-\mathcal{F}(\rho, m)(t)\Big]dt+\mathcal{H}(\rho_1,m_1),
\end{equation}
where the infimum is taken among variables $\rho\colon [0,1]\times \Omega\rightarrow\mathbb{R}_+$, $m\colon [0,1]\times\Omega\rightarrow\mathbb{R}$, $a\colon [0,1]\times \Omega\rightarrow\mathbb{R}$, and $\rho_1\colon\Omega\rightarrow\mathbb{R}_+$, $m_1\colon\Omega\rightarrow\mathbb{R}$ satisfying 
\begin{equation}\label{CBNS}
\left\{\begin{aligned}
&\partial_t\rho(t,x)+\partial_x m(t,x)=0,\\
&\partial_t m(t,x)+\partial_x( \frac{m^2}{\rho})(t,x)+\partial_xP(\rho)(t,x)\\
&\hspace{1.5cm}+\partial_x(\mu(\rho(t,x))a(t,x))=\beta\partial_x(\mu(\rho(t,x))\partial_x\frac{m(t,x)}{\rho(t,x)}),
\end{aligned}\right.
\end{equation}
with fixed initial time value conditions 
\begin{equation*}
\rho(0,x)=\rho_0(x), \qquad m(0,x)=m_0(x). 
\end{equation*}
Here $(\rho_0, m_0)$ is a given pair of functions in $\mathcal{M}$. 
\end{subequations}
\end{definition}
We next derive the critical point system of problem \eqref{CNS} and present its Hamiltonian formalism in metric space $(\mathcal{M}, \g)$. 
\begin{proposition}[Hamiltonian flows of BNS]\label{MFH}
The critical point system of variational problem \eqref{CNS} is given below. There exists a pair of functions $\phi\colon [0,1]\times\Omega\rightarrow\mathbb{R}$ and $\psi\colon[0,1]\times\Omega\rightarrow\mathbb{R}$, such that 
\begin{equation*}
a(t,x)=\partial_x\psi(t,x),
\end{equation*}
and
\begin{equation}\label{CNSeq}
\left\{\begin{aligned}
&\partial_t\rho+\partial_x m=0,\\
&\partial_t m+\partial_x( \frac{m^2}{\rho})+\partial_xP(\rho)+\partial_x(\mu(\rho)\partial_x\psi)=\beta\partial_x(\mu(\rho)\partial_x\frac{m}{\rho}),\\
&\partial_t\phi+\frac{1}{2}|\partial_x\psi|^2\mu'(\rho)-(\frac{m^2}{\rho^2}, \partial_x\psi)+(P'(\rho), \partial_x\psi)+\frac{\delta}{\delta\rho}\mathcal{F}(\rho, m)\\
&\hspace{3cm}=\beta(\partial_x\psi,\partial_x\frac{m}{\rho})\mu'(\rho)+\beta\frac{m}{\rho^2}\partial_x(\mu(\rho)\partial_x\psi),\\
&\partial_t\psi+2\partial_x\psi\cdot\frac{m}{\rho}+\partial_x\phi+\frac{\delta}{\delta m}\mathcal{F}(\rho, m)=-\beta\frac{1}{\rho}\partial_x(\mu(\rho)\partial_x\psi).
\end{aligned}\right.
\end{equation}
Here $'$ represents the derivative w.r.t. variable $\rho$. The initial and terminal time conditions satisfy
\begin{equation*}
\left\{\begin{aligned}
&\rho(0,x)=\rho_0(x), \\
& m(0,x)=m_0(x),\\
&\frac{\delta}{\delta \rho(1,x)}\mathcal{H}(\rho_1,m_1)+\phi(1,x)=0\\
&\frac{\delta}{\delta m(1,x)}\mathcal{H}(\rho_1,m_1)+\psi(1,x)=0.
\end{aligned}\right.
\end{equation*}
\end{proposition}
\begin{proof}
The proof follows the ideas in proving Proposition \ref{MFH} in \cite{LLO}. We present it here for the completeness of this paper. 
Consider a change of variable $w(t,x)=\mu(\rho(t,x))a(t,x)$. In this case, the variational problem \eqref{CNS} is written below.  
\begin{subequations}
\begin{equation}\label{CNS2}
\inf_{\rho, m, w, \rho_1, m_1}~\int_0^1\Big[\int_\Omega \frac{|w(t,x)|^2}{2\mu(\rho(t,x))} dx-\mathcal{F}(\rho, m)(t)\Big]dt+\mathcal{H}(\rho_1,m_1),
\end{equation}
where the infimum is taken among variables $\rho\colon [0,1]\times \Omega\rightarrow\mathbb{R}_+$, $m\colon [0,1]\times\Omega\rightarrow\mathbb{R}$, $w\colon [0,1]\times \Omega\rightarrow\mathbb{R}$, and $\rho_1\colon\Omega\rightarrow\mathbb{R}_+$, $m_1\colon\Omega\rightarrow\mathbb{R}$ satisfying 
\begin{equation}\label{CBNS2}
\left\{\begin{aligned}
&\partial_t\rho+\partial_x m=0,\\
&\partial_t m+\partial_x( \frac{m^2}{\rho})+\partial_xP(\rho)+\partial_x w=\beta\partial_x(\mu(\rho)\partial_x\frac{m}{\rho}),\\
&\rho(0,x)=\rho_0(x), \quad m(0,x)=m_0(x). 
\end{aligned}\right.
\end{equation}
\end{subequations}
We derive the critical point system \eqref{CNSeq} by solving a saddle point problem below. Denote $\phi$, $\Psi\colon [0,1]\times\Omega\rightarrow\mathbb{R}$ as a pair of functions, which are Lagrange multipliers for $\rho$, $m$ in dynamical constraints of $\eqref{CBNS2}$, respectively. Consider 
\begin{equation*}
\inf_{\rho, m, w, \rho_1, m_1}\sup_{\phi,\psi}~~\mathcal{L}(\rho, m, w, \rho_1,m_1, \phi,\psi), 
\end{equation*}
where we define a Lagrangian functional $\mathcal{L}$ as 
\begin{equation*}
\begin{split}
&\mathcal{L}(\rho, m, w, \rho_1,m_1, \phi,\psi)\\=&\quad\int_0^1\Big[\int_\Omega \frac{|w|^2}{2\mu(\rho)} dx-\mathcal{F}(\rho, m)\Big]dt+\mathcal{H}(\rho_1,m_1)\\
&+\int_0^1\int_\Omega \phi\Big(\partial_t\rho+\partial_x m\Big)dxdt\\
&+\int_0^1\int_\Omega\psi\Big(\partial_t m+\partial_x( \frac{m^2}{\rho})+\partial_xP(\rho)+\partial_x w-\beta\partial_x(\mu(\rho)\partial_x\frac{m}{\rho})\Big)dxdt\\
=&\quad\int_0^1\Big[\int_\Omega \frac{1}{2}\frac{|w|^2}{\mu(\rho)} dx-\mathcal{F}(\rho, m)\Big]dt+\mathcal{H}(\rho_1,m_1)\\
&+\int_0^1\int_\Omega \phi \partial_x m+\psi\Big(\partial_x( \frac{m^2}{\rho})+\partial_xP(\rho)+\partial_x w-\beta\partial_x(\mu(\rho)\partial_x\frac{m}{\rho})\Big)dxdt\\
&+\int_\Omega \Big(\phi_1\rho_1+\psi_1m_1\Big) dx-\int_0^1\int_\Omega \Big(\rho\partial_t\phi+m\partial_t\psi\Big) dxdt. 
\end{split}
\end{equation*}
We are now ready to derive the saddle point. Assume $\rho>0$. We let the $L^2$ first variations of $\mathcal{L}$ be zero. In detail, 
\begin{equation*}
\left\{\begin{aligned}
&\frac{\delta}{\delta w}\mathcal{L}=0\\
&\frac{\delta}{\delta \phi}\mathcal{L}=0\\
&\frac{\delta}{\delta \psi}\mathcal{L}=0\\
&\frac{\delta}{\delta\rho}\mathcal{L}=0\\
&\frac{\delta}{\delta m}\mathcal{L}=0\\
&\frac{\delta}{\delta \rho_1}\mathcal{L}=0\\
&\frac{\delta}{\delta m_1}\mathcal{L}=0
\end{aligned}\right.\quad\Rightarrow\quad\left\{\begin{aligned}
&\frac{w}{\mu(\rho)}-\partial_x\psi=0,\\
&\partial_t\rho+\partial_x m=0,\\
&\partial_t m+\partial_x( \frac{m^2}{\rho})+\partial_xP(\rho)+\partial_x w-\beta\partial_x(\mu(\rho)\partial_x\frac{m}{\rho})=0,\\
&-\frac{|w|^2}{2\mu(\rho)^2}\mu'(\rho)-\frac{\delta}{\delta \rho}\mathcal{F}(\rho, m)+(\frac{m^2}{\rho^2}, \partial_x\psi)-(\partial_x\psi, P'(\rho))\\
&+\beta(\partial_x\psi, \partial_x\frac{m}{\rho})\mu'(\rho)+\beta\frac{m}{\rho^2}\partial_x(\mu(\rho)\partial_x\psi)-\partial_t\phi=0,\\
&-\frac{\delta}{\delta m}\mathcal{F}(\rho, m)-\partial_x\phi-(\frac{2m}{\rho}, \partial_x\psi)-\frac{\beta}{\rho}\partial_x(\mu(\rho)\partial_x\psi)-\partial_t\psi=0,\\
&\frac{\delta}{\delta\rho_1}\mathcal{H}(\rho_1, m_1)+\phi_1=0,\\
&\frac{\delta}{\delta m_1}\mathcal{H}(\rho_1, m_1)+\psi_1=0.
\end{aligned}\right.
\end{equation*}
In above formulations, we further use the fact that $\frac{\omega}{\mu(\rho)}=a=\partial_x\psi$. Hence we derive the critical point system \eqref{CNSeq}. 
\end{proof}
\begin{proposition}[Hamiltonian formalisms]
The PDE system \eqref{CNSeq} has the following Hamiltonian flow formulation. 
\begin{equation*}
\left\{\begin{aligned}
\partial_t\rho=&\frac{\delta}{\delta\phi}\mathcal{H}_{\mathcal{G}}(\rho,m, \phi,\psi),\\
\partial_tm=&\frac{\delta}{\delta\psi}\mathcal{H}_{\mathcal{G}}(\rho,m, \phi,\psi),\\
 \partial_t\phi=&-\frac{\delta}{\delta \rho}\mathcal{H}_{\mathcal{G}}(\rho,m, \phi,\psi),\\
  \partial_t\psi=&-\frac{\delta}{\delta m}\mathcal{H}_{\mathcal{G}}(\rho,m, \phi,\psi),\\
\end{aligned}\right.
\end{equation*}
where we define a Hamiltonian functional $\mathcal{H}_{\mathcal{G}}$ as 
\begin{equation}\label{eq:hamiltonian_functional}
\begin{split}
&\mathcal{H}_{\mathcal{G}}(\rho,m, \phi, \psi)\\
=&\int_\Omega \Big[\frac{1}{2}(\partial_x\psi,\partial_x\psi)\mu(\rho)+(m, \partial_x\phi)+(\frac{m^2}{\rho}+P(\rho), \partial_x\psi)-\beta (\partial_x\psi, \partial_x\frac{m}{\rho})\mu(\rho) \Big]dx+\mathcal{F}(\rho, m).
\end{split}
\end{equation}
\end{proposition}

\begin{proposition}[Functional Hamilton-Jacobi equation of BNS]
The Hamilton-Jacobi equation in $(\mathcal{M}, \g)$ satisfies
\begin{equation*}
\begin{split}
&\partial_t\mathcal{U}(t,\rho, m)+\frac{1}{2}\int_\Omega\big(\partial_x\frac{\delta}{\delta m(x)}\mathcal{U}(t,\rho, m), \partial_x\frac{\delta}{\delta m(x)}\mathcal{U}(t,\rho,m)\big)\mu(\rho(x))dx\\
&\hspace{1.9cm}+\int_\Omega\big(\partial_x\frac{\delta}{\delta \rho(x)}\mathcal{U}(t,\rho, m), m(x)\big)dx+\int_\Omega\big(\partial_x\frac{\delta}{\delta m(x)}\mathcal{U}(t,\rho, m), \frac{m(x)^2}{\rho(x)}+P(\rho(x))\big)dx\\
&\hspace{1.9cm}-\beta \int_\Omega(\partial_x\frac{\delta}{\delta m(x)}\mathcal{U}(t,\rho, m), \partial_x\frac{m(x)}{\rho(x)})\mu(\rho(x))dx+\mathcal{F}(\rho, m)=0,
\end{split}
\end{equation*}
where $\mathcal{U}\colon [0,1]\times L^2(\Omega)\times L^2(\Omega)\rightarrow\mathbb{R}$ is a value functional. 
\end{proposition}
\begin{proof}
We only need to prove that equation is an Hamiltonian flow in $(\mathcal{M}, \g)$. We can check it directly by computing the $L^2$ first order variations of the Hamiltonian functional $\mathcal{H}_{\mathcal{G}}$ w.r.t. variables $\rho, m, \phi, \psi$, respectively. Clearly, 
\begin{equation*}
\left\{\begin{aligned}
&\frac{\delta}{\delta\phi}\mathcal{H}_{\mathcal{G}}(\rho,m,\phi,\psi)=-\partial_x m,\\
&\frac{\delta}{\delta\psi}\mathcal{H}_{\mathcal{G}}(\rho,m,\phi,\psi)=-\partial_x(\frac{m^2}{\rho}+P(\rho))-\beta\partial_x(\mu(\rho)\partial_x\psi)+\beta\partial_x(\mu(\rho)\partial_x\frac{m}{\rho}),\\
&\frac{\delta}{\delta\rho}\mathcal{H}_{\mathcal{G}}(\rho,m,\phi,\psi)=-\frac{m^2}{\rho^2}\partial_x\psi+P'(\rho)\partial_x\psi+\frac{1}{2}|\partial_x\psi|^2\mu'(\rho)-\beta\partial_x(\mu(\rho)\partial_x\psi)\frac{m}{\rho^2}\\
&\hspace{4.8cm}-\beta(\partial_x\psi, \partial_x\frac{m}{\rho^2})\mu'(\rho)+\frac{\delta}{\delta\rho}\mathcal{F}(\rho, m),\\
&\frac{\delta}{\delta m}\mathcal{H}_{\mathcal{G}}(\rho, m,\phi,\psi)=\partial_x\phi+\frac{2m}{\rho}\partial_x\psi+\frac{\beta}{\rho}\partial_x(\mu(\rho)\partial_x\psi)+\frac{\delta}{\delta m}\mathcal{F}(\rho, m).
\end{aligned}\right.
\end{equation*}
In addition, the Hamilton-Jacobi equation in $(\mathcal{M}, \g)$ satisfies 
\begin{equation*}
\partial_t\mathcal{U}(t,\rho, m)+\mathcal{H}_{\mathcal{G}}(\rho, m, \frac{\delta}{\delta \rho}\mathcal{U}(t,\rho, m), \frac{\delta}{\delta m}\mathcal{U}(t,\rho, m))=0,
\end{equation*}
where $\frac{\delta}{\delta\rho}$, $\frac{\delta}{\delta m}$ are first variation operators w.r.t. $\rho$, $m$, respectively. This finishes the derivation. 
\end{proof}
\subsection{Examples}
In this subsection, we present several examples of control problems of BNS \eqref{CNS} and the primal--dual BNS \eqref{CNSeq}. 
\begin{example}[$\alpha=1$]
Consider $\mu(\rho)=\rho$. In this case, variational problem \eqref{CNS} forms 
\begin{equation*}
\inf_{\rho, m, a, \rho_1, m_1}~\int_0^1\Big[\int_\Omega \frac{1}{2}|a(t,x)|^2\rho(t,x)dx-\mathcal{F}(\rho, m)(t)\Big]dt+\mathcal{H}(\rho_1,m_1),
\end{equation*}
s.t.
\begin{equation*}
\left\{\begin{aligned}
&\partial_t\rho+\partial_x m=0,\\
&\partial_t m+\partial_x( \frac{m^2}{\rho})+\partial_xP(\rho)+\partial_x(\rho a)=\beta\partial_x(\rho\partial_x\frac{m}{\rho}),\\
&\rho(0,x)=\rho_0(x),~~m(0,x)=m_0(x). 
\end{aligned}\right.
\end{equation*}
The critical point system of above minimizer problem satisfies 
\begin{equation*}
\left\{\begin{aligned}
&\partial_t\rho+\partial_x m=0,\\
&\partial_t m+\partial_x( \frac{m^2}{\rho})+\partial_xP(\rho)+\partial_x(\rho\partial_x\psi)=\beta\partial_x(\rho\partial_x\frac{m}{\rho}),\\
&\partial_t\phi+\frac{1}{2}|\partial_x\psi|^2-(\frac{m^2}{\rho^2}, \partial_x\psi)+(P'(\rho), \partial_x\psi)+\frac{\delta}{\delta\rho}\mathcal{F}(\rho, m)=\beta(\partial_x\psi,\partial_x\frac{m}{\rho})+\beta\frac{m}{\rho^2}\partial_x(\rho\partial_x\psi),\\
&\partial_t\psi+2\partial_x\psi\cdot\frac{m}{\rho}+\partial_x\phi+\frac{\delta}{\delta m}\mathcal{F}(\rho, m)=-\beta\frac{1}{\rho}\partial_x(\rho\partial_x\psi).
\end{aligned}\right.
\end{equation*}
 In other words, 
 \begin{equation*}
 \partial_t\rho=\frac{\delta}{\delta\phi}\mathcal{H}_{\mathcal{G}},~~ 
\partial_tm=\frac{\delta}{\delta\psi}\mathcal{H}_{\mathcal{G}},~~
 \partial_t\phi=-\frac{\delta}{\delta \rho}\mathcal{H}_{\mathcal{G}},~~
  \partial_t\psi=-\frac{\delta}{\delta m}\mathcal{H}_{\mathcal{G}},
 \end{equation*} 
 where the Hamiltonian functional $\mathcal{H}_{\mathcal{G}}$ satisfies  
\begin{equation*}\begin{split}
&\mathcal{H}_{\mathcal{G}}(\rho,m, \phi, \psi)\\
=&\int_\Omega \Big[\frac{1}{2}(\partial_x\psi,\partial_x\psi)\rho+(m, \partial_x\phi)+(\frac{m^2}{\rho}+P(\rho), \partial_x\psi)-\beta (\partial_x\psi, \partial_x\frac{m}{\rho})\rho\Big]dx+\mathcal{F}(\rho, m).
\end{split}
\end{equation*}
\end{example}
\begin{example}[$\alpha=0$]
Consider $\mu(\rho)=1$. In this case, variational problem \eqref{CNS} forms 
\begin{equation*}
\inf_{\rho, m, a, \rho_1, m_1}~\int_0^1\Big[\int_\Omega \frac{1}{2}|a(t,x)|^2dx-\mathcal{F}(\rho, m)(t)\Big]dt+\mathcal{H}(\rho_1,m_1),
\end{equation*}
s.t.
\begin{equation*}
\left\{\begin{aligned}
&\partial_t\rho+\partial_x m=0,\\
&\partial_t m+\partial_x( \frac{m^2}{\rho})+\partial_xP(\rho)+\partial_x a=\beta\partial_x(\partial_x\frac{m}{\rho}),\\
&\rho(0,x)=\rho_0(x),~~m(0,x)=m_0(x). 
\end{aligned}\right.
\end{equation*}
The critical point system of above variational problem satisfies 
\begin{equation*}\label{CNSeq}
\left\{\begin{aligned}
&\partial_t\rho+\partial_x m=0,\\
&\partial_t m+\partial_x( \frac{m^2}{\rho})+\partial_xP(\rho)+\partial_x(\partial_x\psi)=\beta\partial_x(\partial_x\frac{m}{\rho}),\\
&\partial_t\phi-(\frac{m^2}{\rho^2}, \partial_x\psi)+(P'(\rho), \partial_x\psi)+\frac{\delta}{\delta\rho}\mathcal{F}(\rho, m)=\beta\frac{m}{\rho^2}\partial_x(\partial_x\psi),\\
&\partial_t\psi+2\partial_x\psi\cdot\frac{m}{\rho}+\partial_x\phi+\frac{\delta}{\delta m}\mathcal{F}(\rho, m)=-\beta\frac{1}{\rho}\partial_x(\partial_x\psi).
\end{aligned}\right.
\end{equation*}
 In other words, 
 \begin{equation*}
 \partial_t\rho=\frac{\delta}{\delta\phi}\mathcal{H}_{\mathcal{G}},~~ 
\partial_tm=\frac{\delta}{\delta\psi}\mathcal{H}_{\mathcal{G}},~~
 \partial_t\phi=-\frac{\delta}{\delta \rho}\mathcal{H}_{\mathcal{G}},~~
  \partial_t\psi=-\frac{\delta}{\delta m}\mathcal{H}_{\mathcal{G}},
 \end{equation*} 
 where the Hamiltonian functional $\mathcal{H}_{\mathcal{G}}$ satisfies  
\begin{equation*}\begin{split}
&\mathcal{H}_{\mathcal{G}}(\rho,m, \phi, \psi)\\
=&\int_\Omega \Big[\frac{1}{2}(\partial_x\psi,\partial_x\psi)+(m, \partial_x\phi)+(\frac{m^2}{\rho}+P(\rho), \partial_x\psi)-\beta (\partial_x\psi, \partial_x\frac{m}{\rho})\Big]dx+\mathcal{F}(\rho, m).
\end{split}
\end{equation*}
\end{example}

\section{Numerical methods and examples}\label{section4}
This section designs numerical schemes for optimal control of barotropic compressible Euler equations in $1$D. It proposes an algorithm inspired by the primal-dual hybrid gradient method (PDHG) to solve the control problem.

\subsection{The PDHG inspired algorithm}
The primal-dual hybrid gradient algorithm \cite{champock11} solves the saddle-point problem 
\begin{align*}
\min_z \max_p \langle Kz,p\rangle_{L^2} + g(z) -h^*(p),
\end{align*}
where $\mathcal{Z}$ is a finite or infinite dimensional Hilbert space, $h$ and $g$ are convex functions and $K:\mathcal{Z}\rightarrow \mathcal{H}$ is a linear operator between Hilbert spaces. The function $h^*$ is the convex conjugate of $h$, where $h^*(p) = \sup_z \langle Kz,p\rangle_{L^2} -h(z)$. 
The algorithm solves the saddle-point problem by iterating the following steps:
\begin{align*}
z^{n+1} &= \argmin_{z}   \langle Kz,\Tilde{p}^n\rangle_{L^2}  + g(z) +  \frac{1}{2 \tau} \| z - z^n\|^2_{L^2},\\
p^{n+1} &= \argmax_{p}   \langle Kz^{n+1},p\rangle_{L^2}  -h^*(p) -\frac{1}{2 \sigma} \| p - p^n\|^2_{L^2},\\
\Tilde{p}^{n+1} &= 2p^{n+1} - p^{n}.
\end{align*}
Here $\tau$($\sigma$) is the stepsize for proximal gradient descent(ascent) steps respectively. The algorithm converges if $\sigma \tau \|K^T K\|<1 $. There are various extenstions of PDHG, including nonlinear PDHG \cite{clason2017primal} where the operator $K$ is nonlinear and the General-proximal Primal-Dual Hybrid Gradient (G-prox PDHG) method \cite{JacobsLegerLiOsher2018_solvinga} where choosing proper norms ($L^2, H^1, ...$) for the proximal step allows larger stepsizes.


Inspired by the PDHG method and its variants, we use the saddle point formulation of the optimal control of BNS \eqref{CNS} and propose an algorithm to solve it. Denote
	\begin{align*}
z & = (\rho,m,a,\rho_1,m_1),\\
p & = (\phi,\psi),\\
K\left(\rho,m,a,\rho_1,m_1\right) & = \begin{pmatrix}
\partial_t\rho+\partial_x m\\
\partial_t m+\partial_x( \frac{m^2}{\rho})+\partial_xP(\rho) +\partial_x(\mu(\rho)a) - \beta\partial_x(\mu(\rho))\partial_x\frac{m}{\rho})
\end{pmatrix} , \\
g\left(\rho,m,a,\rho_1,m_1\right) & = \int_0^1\Big[\int_\Omega \frac{1}{2}|a(t,x)|^2\mu(\rho(t,x)) dx-\mathcal{F}(\rho, m)(t)\Big]dt+\mathcal{H}(\rho_1,m_1),\\
h(Kz) & = \begin{cases}
0 \quad \text{if} \;Kz = 0\\
+\infty \quad \text{else}
\end{cases}.
\end{align*}
The corresponding inf-sup problem takes the following form

\begin{equation}\label{eqn:minmax0}
	\inf_{\rho,m,a,\rho_1,m_1} \sup_{\phi,\psi}~~\mathcal{L}(\rho,m,a,\rho_1,m_1,\phi,\psi),
	\end{equation}
	subject to
	\begin{equation*}
\left\{\begin{aligned}
&\rho(0,x)=\rho_0(x), \quad
 m(0,x)=m_0(x),\\
&\frac{\delta}{\delta \rho(1,x)}\mathcal{H}(\rho_1,m_1)+\phi(1,x)=0,\quad
\frac{\delta}{\delta m(1,x)}\mathcal{H}(\rho_1,m_1)+\psi(1,x)=0,
\end{aligned}\right.	    
\end{equation*}
where
	\begin{equation}\label{eq:L_pdhg_cont}
	\begin{aligned}{}
&	\mathcal{L}(\rho,m,a,\rho_1,m_1,\phi,\psi) \\ =&\quad\int_0^1\Big[\int_\Omega \frac{1}{2}|a(t,x)|^2\mu(\rho(t,x)) dx-\mathcal{F}(\rho, m)(t)\Big]dt+\mathcal{H}(\rho_1,m_1) \\
&	+ \int_0^1\int_{\Omega} \phi \left( \partial_t\rho+\partial_x m\right)  dxdt\\
&+ \int_0^1\int_\Omega\psi \left(\partial_t m+\partial_x( \frac{m^2}{\rho})+\partial_x P(\rho) +\partial_x(\mu(\rho)a) - \beta\partial_x(\mu(\rho)\partial_x\frac{m}{\rho})\right)  dx dt.
		\end{aligned}
\end{equation}

We choose $L^2$ norm for primal variable $(\rho,m,a)$ update and $H$ norm for $(\phi,\psi)$, where
\begin{align*}
    \|v\|^2_{L^2} =\int_0^1 \int_{\Omega}v^2 dxdt,\quad \|v\|^2_{H} =  c_1\|\nabla v\|^2_{L^2} + c_2\|\Delta v\|^2_{L^2} + c_3\|\partial_t v\|^2_{L^2}. 
\end{align*}
Here the parameters $c_i, i= 1,2,3$ are chosen based on the operator $K$. 

We now present the algorithm as follows.

\begin{algorithm}
\caption{Algorithm 1: PDHG for optimal control of BNS
}\label{alg:short}
\begin{flushleft}
    \hspace*{\algorithmicindent} \textbf{Input:} A set of initial guess of $(\rho,m,a,\rho_1,m_1,\phi,\psi)$\\
    \hspace*{\algorithmicindent} \textbf{Output:} $(\rho,m,a,\rho_1,m_1)$
    \end{flushleft}
\begin{algorithmic}
    \While {iteration  $k<\mathcal{K}_{\text{maximal}}$}
    \State{ \begin{align*}
        &\left(\rho^{(k+1)},m^{(k+1)},a^{(k+1)},\rho_1^{(k+1)},m_1^{(k+1)}\right) \\&= \argmin\limits_{\rho,m,a,\rho_1,m_1}~~ \mathcal{L}(\rho,m,a,\rho_1,m_1,\Tilde{\phi}^k,\Tilde{\psi}^k) + \frac{1}{2\tau} \|\rho- \rho^{(k)}\|^2_{L^2} + \frac{1}{2\tau} \|m - m^{(k)}\|^2_{L^2}+ \frac{1}{2\tau} \|a - a^{(k)}\|^2_{L^2} \\& + \frac{1}{2\tau} \|\rho_1 - \rho_1^{(k)}\|^2_{L^2}+ \frac{1}{2\tau} \|m_1 - m_1^{(k)}\|^2_{L^2};
    \end{align*}}
\State{ \begin{align*}
   & \left(\phi^{(k+ 1)},\psi^{(k+ 1)}\right) \\ & = \argmax\limits_{\phi,\psi}~~\mathcal{L}(\rho^{(k+1)},m^{(k+1)},a^{(k+1)},\rho_1^{(k+1)},m_1^{(k+1)},{\phi},{\psi})  - \frac{1}{2\sigma} \|\phi - \phi^{(k)}\|^2_{H}  - \frac{1}{2\sigma} \|\psi - \psi^{(k)}\|^2_{H};\\
       & \left(\Tilde{\phi}^{(k+1)} ,\Tilde{\psi}^{(k+1)} \right) = \left(2 \phi^{(k+1)} - \phi^{(k)},2 \psi^{(k+1)} - \psi^{(k)}\right);
    \end{align*}
    }
    \State{$k \gets k+1$;}
    \EndWhile 
\end{algorithmic}
\end{algorithm}

\subsection{Finite Difference Scheme of the control problem in the variational form.}
We consider the barotropic compressible Euler equation and discretize it using Lax--Friedrichs type of scheme. Consider the domain $[0,1]\times [0,1] $ in space-time. Given $N_{x},N_t>0$, we have $\Delta x = \frac{1}{N_{x}}$, $\Delta t = \frac{1}{N_{t}}$.
For $x_i = i \Delta x,t_l =l \Delta t $, define 
\begin{align*}
u_{i}^l &= u(t_l,x_i),\\
D_c(u)_{i} &= \frac{u_{i+1} -  u_{i-1}}{2 \Delta x},\quad Lap(u)_{i} = \frac{u_{i+1} - 2u_{i} + u_{i-1}}{(\Delta x)^2},\\
D(a(Du))_{i}^{l+1} & =  \frac{1}{\Delta x^2} \left(\frac{a_{i+1}^{l+1} +a_{i}^{l+1}}{2}\left(u_{i+1}^{l+1} - u_{i}^{l+1}\right) -\frac{a_{i}^{l+1} +a_{i-1}^{l+1}}{2}\left(u_{i}^{l+1} - u_{i-1}^{l+1}\right) \right)
\end{align*}

The barotropic compressible Euler equation adapted from the Lax--Friedrichs scheme is as follows:
	\begin{equation}\label{eq:density_discrete}
	\begin{aligned}
	&\frac{1}{\Delta t} \left( \rho_{i}^{l+1} - \rho_{i}^{l}\right) + D_c(m)_i^{l+1}  -c\Delta x Lap(\rho)^{l+1}_i = 0,
		\end{aligned}
	\end{equation}
		\begin{equation}\label{eq:momentum_discrete}
	\begin{aligned}
	&		\frac{1}{\Delta t} \left( m_{i}^{l+1} - m_{i}^{l}\right) + D_c(\frac{m^2}{\rho})_{i}^{l+1} + D_c(P(\rho))_{i}^{l+1}  + D_c(\mu(\rho))_{i}^{l+1} a_{i}^{l+1} 
-\beta D\left(\mu(\rho) D (\frac{m}{\rho})\right)_{i}^{l+1}\\ &	-  c' \Delta x Lap(m)^{l+1}_i = 0,
	\end{aligned}
	\end{equation}
	for $1\leq i \leq N_{x}$, $0\leq l \leq N_t-1$. And $c, c'>0$ are artificial viscosity coefficients. We use the implicit scheme that fits the feedback structure of the optimal control problem. The discrete min-max problem is as follows:
\begin{align*}
\min_{\rho,m,a,\rho_1,m_1} \max_{\phi,\psi} L(\rho,m,a,\rho_1,m_1,\phi,\psi),
\end{align*}
where 
\begin{align*}
&   L(\rho,m,a,\rho_1,m_1,\phi,\psi) \\= & \Delta x \Delta t \sum_{\substack{1 \leq i \leq N_x\\1 \leq l \leq N_t}} |a_i^l|^2\mu(\rho_i^l)-\Delta t \sum_{ 1 \leq l \leq N_t} \mathcal{F}(\rho^l,m^l) + \Delta x \sum_{1 \leq i \leq N_x}\mathcal{H}( \rho_{i}^{N_t},m_{i}^{N_t}) \\
& + \Delta x \Delta t \sum_{\substack{1 \leq i \leq N_x\\ 0 \leq l \leq N_t-1}} \Bigg\{\phi_{i}^{l} \left(  \frac{1}{\Delta t} \left( \rho_{i}^{l+1} - \rho_{i}^{l}\right) + D_c(m)_i^{l+1}  -c\Delta x Lap(\rho)^{l+1}_i \right) \\&\hspace{3cm}+ \psi_{i}^{l}  \Bigg( \frac{1}{\Delta t} \left( m_{i}^{l+1} - m_{i}^{l}\right) + D_c(\frac{m^2}{\rho})_{i}^{l+1} + D_c(P(\rho))_{i}^{l+1}  \\
&\hspace{2.5cm}+ D_c(\mu(\rho))_{i}^{l+1}a_{i}^{l+1} 
-\beta D\left(\mu(\rho) D (\frac{m}{\rho})\right)_{i}^{l+1}- c' \Delta x Lap(m)^{l+1}_i \Bigg)\Bigg\}.
\end{align*}

Via the summation by parts and take first order variational derivative, we derive the implicit finite difference scheme for the dual equations of $\phi,\psi$.
\begin{equation}
    \begin{aligned}
    &	\frac{1}{\Delta t} \left( \phi_i^{l+1} - \phi_i^{l}\right) + \frac{1}{2}\left(D_c(\psi)_i^l\right)^2\mu'(\rho_i^l) + D_c(\psi)_i^l\left( P'(\rho_i^l)- \left(\frac{m_i^l}{\rho_i^l}\right)^2\right) \\
    &+ \mu'(\rho_i^l)\frac{\psi_i^l - \psi_{i-1}^l}{2 \Delta x} \left( \frac{m_{i}^{l+1}}{\rho_{i}^{l+1}} - \frac{m_{i-1}^{l+1}}{\rho_{i-1}^{l+1}}\right)  + \mu'(\rho_i^l)\frac{\psi_i^l - \psi_{i-1}^l}{2 \Delta x} \left( \frac{m_{i+1}^{l+1}}{\rho_{i+1}^{l+1}} - \frac{m_{i}^{l+1}}{\rho_{i}^{l+1}}\right) +\frac{\delta \mathcal{F}(\rho_i^l,m_i^l)}{\delta \rho}\\& =\beta \frac{m_i^l}{(\rho_i^l)^2} (c\Delta x) Lap(\phi)^{l}_i,
    \end{aligned}
\end{equation}
and 
\begin{equation}
    \begin{aligned}
    &	\frac{1}{\Delta t} \left( \psi_i^{l+1} - \psi_i^{l}\right) + 2 D_c(\psi)_i^l\frac{m_i^l}{\rho_i^l} + D_c(\phi)_i^l +\frac{\delta \mathcal{F}(\rho_i^l,m_i^l)}{\delta m}
 +\beta \frac{1}{\rho_i^l}D\left(\mu(\rho) D (\psi)\right)_{i}^{l}\\ &=  (c'\Delta x) Lap(\psi)^{l}_i.
    \end{aligned}
\end{equation}
\subsection{Numerical examples}
We provide three examples here to illustrate the proposed control problem. Without further specification, examples are considered in $[0,1] \times [0,1]$ in space-time domain. The spatial domain is imposed with periodic boundary condition. We have uniform mesh size in space and time, with $\Delta t = \frac{1}{N_t}, \Delta x = \frac{1}{N_x}$, $N_t = 32, N_x = 64$. We set the iteration number $\mathcal{K}_{maximal} = 5 \cdot 10^4$, and the stepsizes of $\tau,\sigma $ are tuned in each example.

\subsection{Example 1}
In the first example, we consider a degenerate case where there is essentially no control, i.e., $\mathcal{F} = 0, \mathcal{H} =0$. Solving this control problem is equivalent to solving an initial-value problem of BNS system. We set initial condition as follows with discontinuous piece-wise constant:
\begin{align*}
 \rho_0(x) =
    \begin{cases}
    2 \quad &\text{if}\; 0.25 < x < 0.75\\
    1 \quad &\text{else} \\
    \end{cases},
  \quad m_0(x) =
    \begin{cases}
    1 \quad &\text{if}\; 0.25 < x < 0.75\\
    0.5 \quad &\text{else} \\
    \end{cases}.   
\end{align*}
We consider this problem in $[0,1] \times [0,0.2]$ space-time domain, with mesh $N_x = 64, N_t = 16$.
To verify that our proposed model solves the initial-value problem of BNS system, we compare the result with a forward explicit finite difference scheme of the BNS system:
\begin{align*}
\begin{cases}
    	&\frac{\left( \rho_{i}^{l+1} - \rho_{i}^{l}\right)}{\Delta t}  + D_c(m)_i^{l}  -c\Delta x Lap(\rho)^{l}_i = 0,\\
    		&		\frac{\left( m_{i}^{l+1} - m_{i}^{l}\right)}{\Delta t}  + D_c(\frac{m^2}{\rho})_{i}^{l} + D_c(P(\rho))_{i}^{l}  
-\beta D\left(\mu(\rho) D (\frac{m}{\rho})\right)_{i}^{l}	-  c' \Delta x Lap(m)^{l}_i = 0.
\end{cases}
\end{align*}
The explicit scheme needs to satisfy the CFL condition, which leads to a very fine mesh in time. In this example, we set $N_x = 64, N_t = 256$.
The BNS system has $\mu(\rho)= 1, P(\rho) =0.1 \rho^2, \beta = 0.1, c=0.5, c'=0.5.$
\begin{figure}
    \centering
    \includegraphics[width=0.9\textwidth]{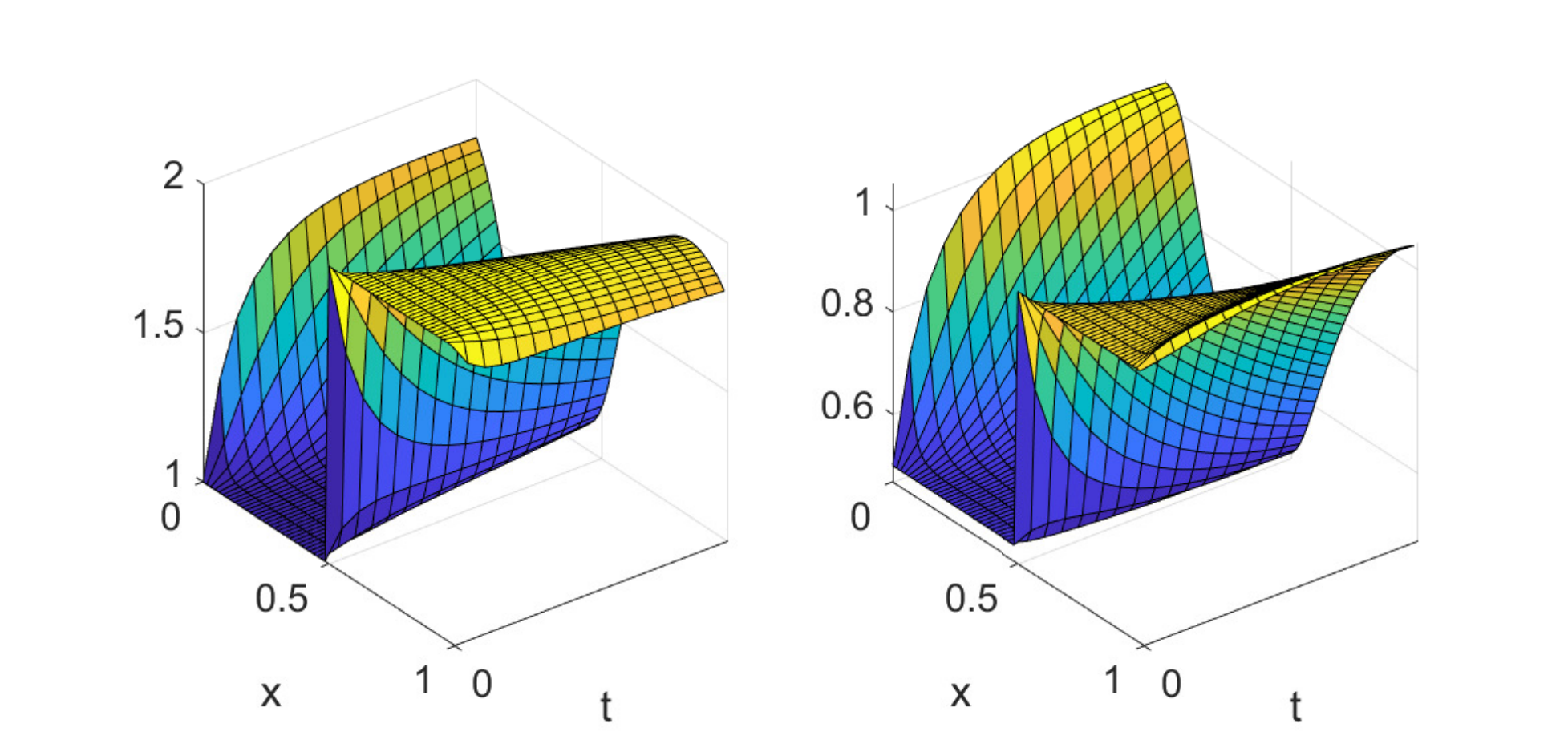}
    \caption{Solution to the BNS equation via control problem in example 1: $\rho(x,t)$ (left); $m(x,t)$ (right).}
    \label{fig:example0_2d}
    \centering
    \includegraphics[width=1\textwidth]{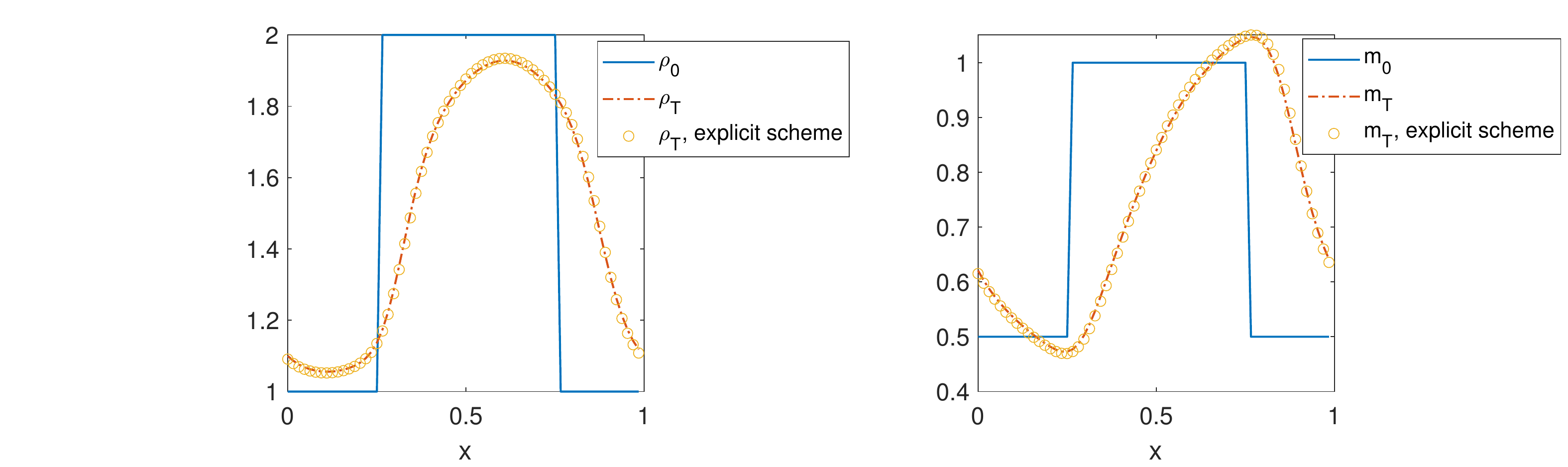}
    \caption{Initial condition of the BNS equation ($\rho_0, m_0$) in example 1 and the comparison of two solution via solving an optimal control problem ($\rho_T, m_T$) and using explicit scheme ($\rho_T, m_T$ explicit scheme) at final time $T= 0.2$.}
    \label{fig:example0_1d}
\end{figure}
The numerical results from Figure \ref{fig:example0_2d},\ref{fig:example0_1d} shows that our optimal control problem can successfully recover the initial value problem for the BNS system. Thanks to the implicit finite difference scheme, the optimal control problem allows larger step sizes in time. 
We expect that the computational complexity of our primal--dual approach will be lower than the explicit finite difference schemes as we refine the grid.

\subsubsection{Example 2} We consider a control problem of the BNS system where $\mu(\rho) = 1, P(\rho) =0.1 \rho^2,\beta = 0.1$. Numerical artificial viscosity $c = 0.5, c' = 0.5.$ The initial conditions for density and momentum are
\begin{align*}
    \rho_0(x) = 0.1 + 0.9 \exp(-100(x-0.5)^2), \quad m_0(x) = 0.
\end{align*}
 As for the control problem, we set $\mathcal{F} = 0, \mathcal{H}(\rho_1,m_1) = \int_{\Omega} \rho_1(x) g(x) dx$. We test two cases: $g_1(x) = 0, g_2(x) = -0.1\exp(-100(x-0.25)^2)$. In the first case, the optimal control problem will degenerate to the BNS equations without control; the solution $\rho,m$ will correspond to the original initial value problem. As for the second case, the final cost functional $\mathcal{H}$ we choose will make density concentrate around $x = 0.25$. 

\begin{figure}
        \centering
 
        {\includegraphics[width=0.9\textwidth]{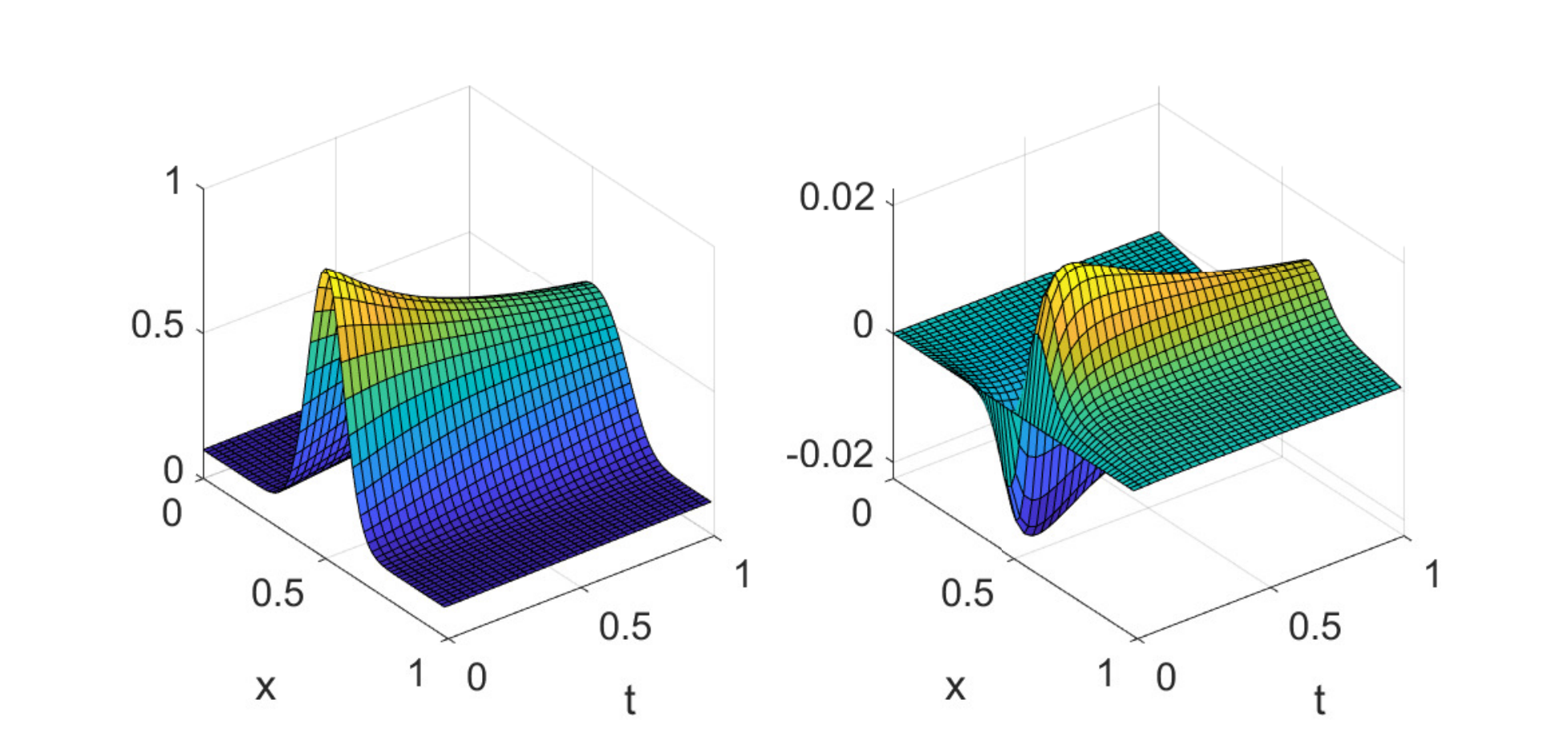} } 
        \caption{Numerical results: the density (left) and the momentum (right) over time for case $g_1 = 0$ in example 2.}
               \label{fig:eg0_g1}
    \hfill
        \centering

        \includegraphics[width=0.9\textwidth]{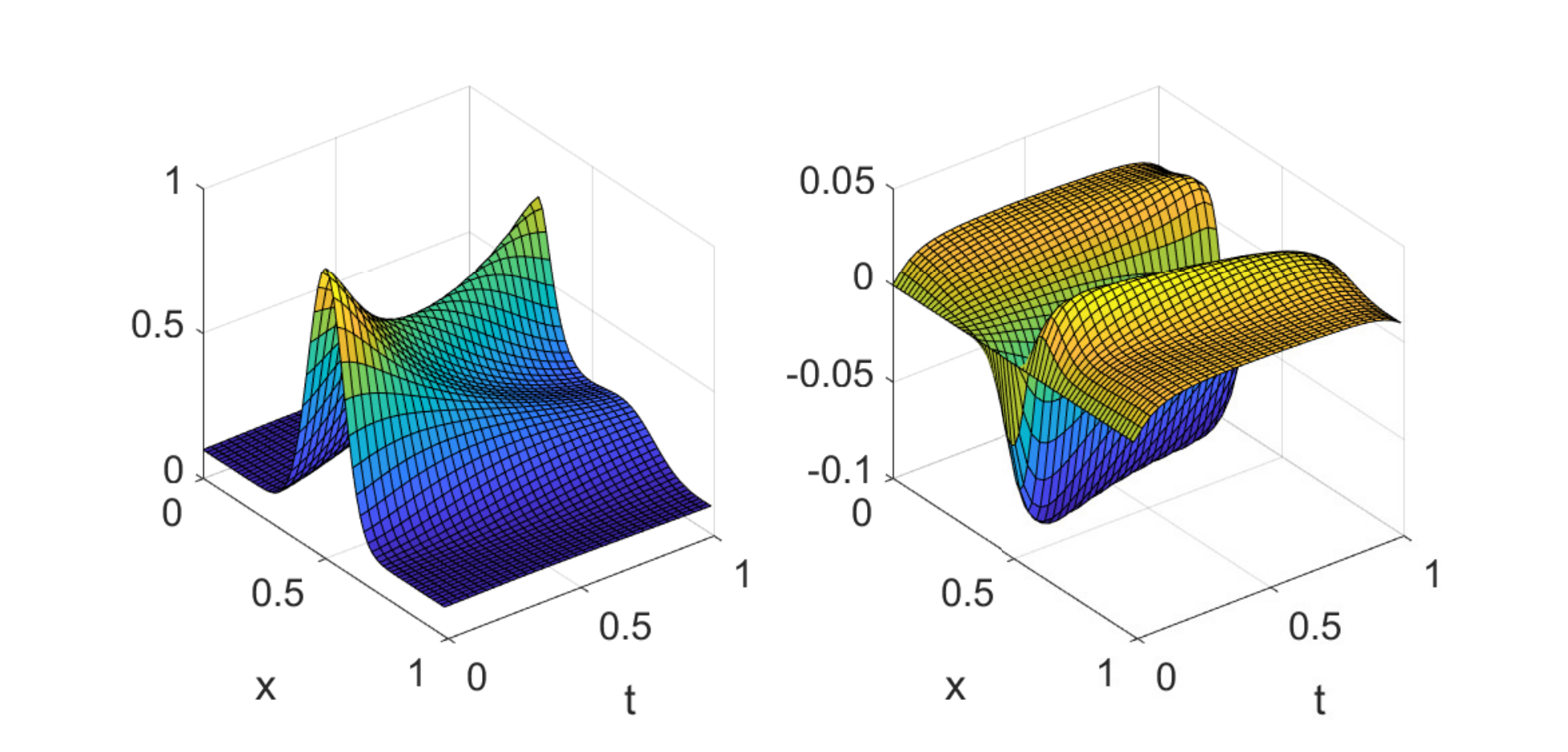} 
        \caption{Numerical results: the density (left) and the momentum (right)  over time for case $g_2 = -0.1\exp(-100(x-0.25)^2)$ in example 2.}
                 \label{fig:eg0_g2}
\end{figure}
We can see from the numerical result in Figure \ref{fig:eg0_g1}, \ref{fig:eg0_g2} that with $\mathcal{H} = 0$, the density only diffuses in the first case; while in the second case a final cost functional is imposed at terminal time, the density moves towards $x = 0.25$ enforced by external control (from $a$). 

\subsubsection{Example 3}
 We consider a control problem of the BNS system where $\mu(\rho) = \rho, P(\rho) =0.1 \rho^2, \beta = 0.1$. Numerical artificial viscosity $c = 0.1, c' = 0.$ The initial conditions for density and momentum are
\begin{align*}
    \rho_0(x) = 1 + \exp(-100(x-0.5)^2), \quad m_0(x) = 0.
\end{align*}
We set $\mathcal{F}(\rho,m) = \int_{\Omega} c_F m^2 dx ,\; \mathcal{H}(\rho_1,m_1) = \int_{\Omega} \rho_1(x) g(x) dx$, where $g(x) = 0.1 \sin(4\pi x)$. Similarly to the first example, the final cost functional makes the density move towards $x=\frac{3}{8}, \frac{7}{8}$. The term $\mathcal{F}(\rho, m)$ penalize the control system with large momentum for $c_F>0$. 

\begin{figure}
        \centering
        {\includegraphics[width=0.9\textwidth]{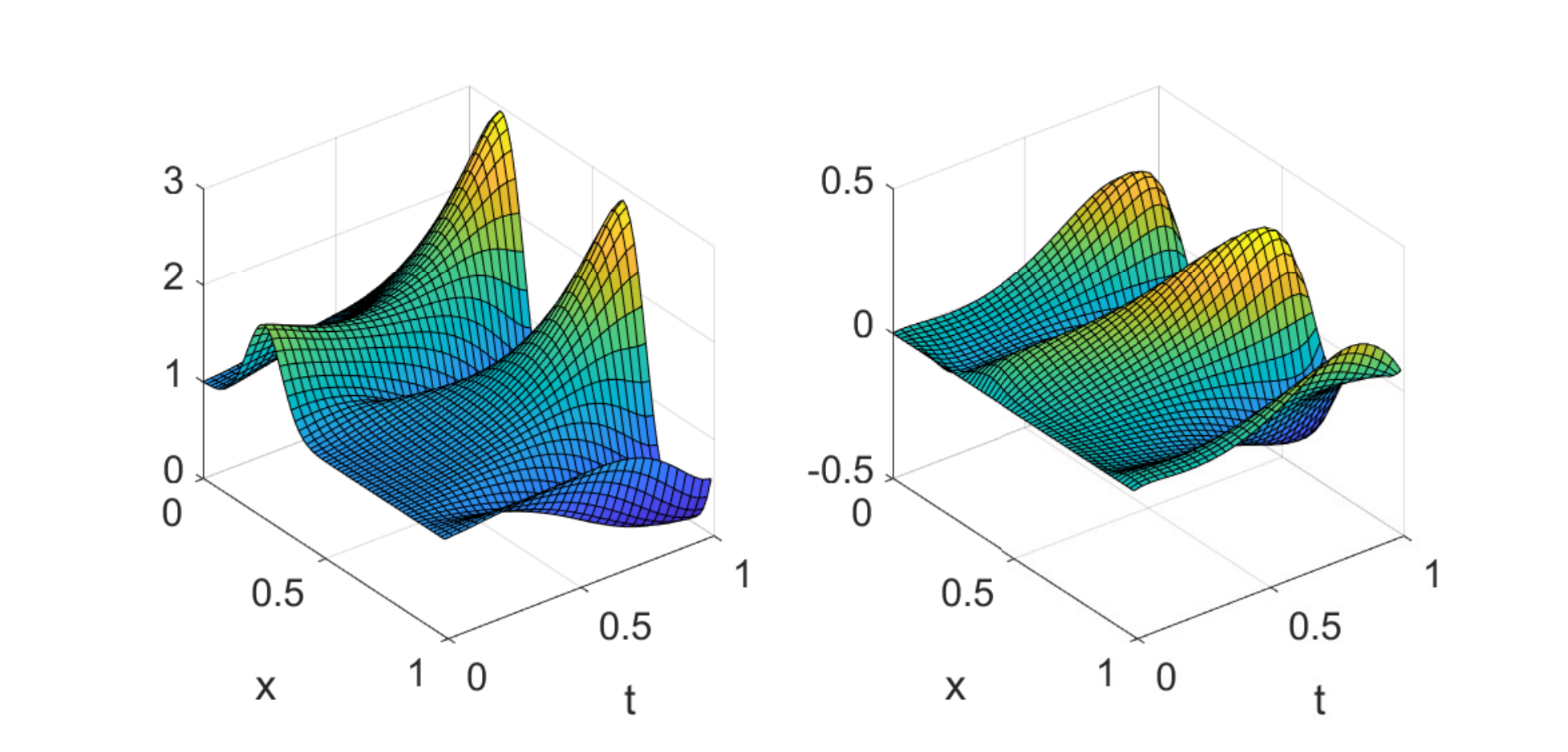} } 
        \caption{The density (left) and the momentum (right) change over time for case $c_F = 0$ in example 3.}
               \label{fig:eg2a}
    \hfill
        \centering

        \includegraphics[width=0.9\textwidth]{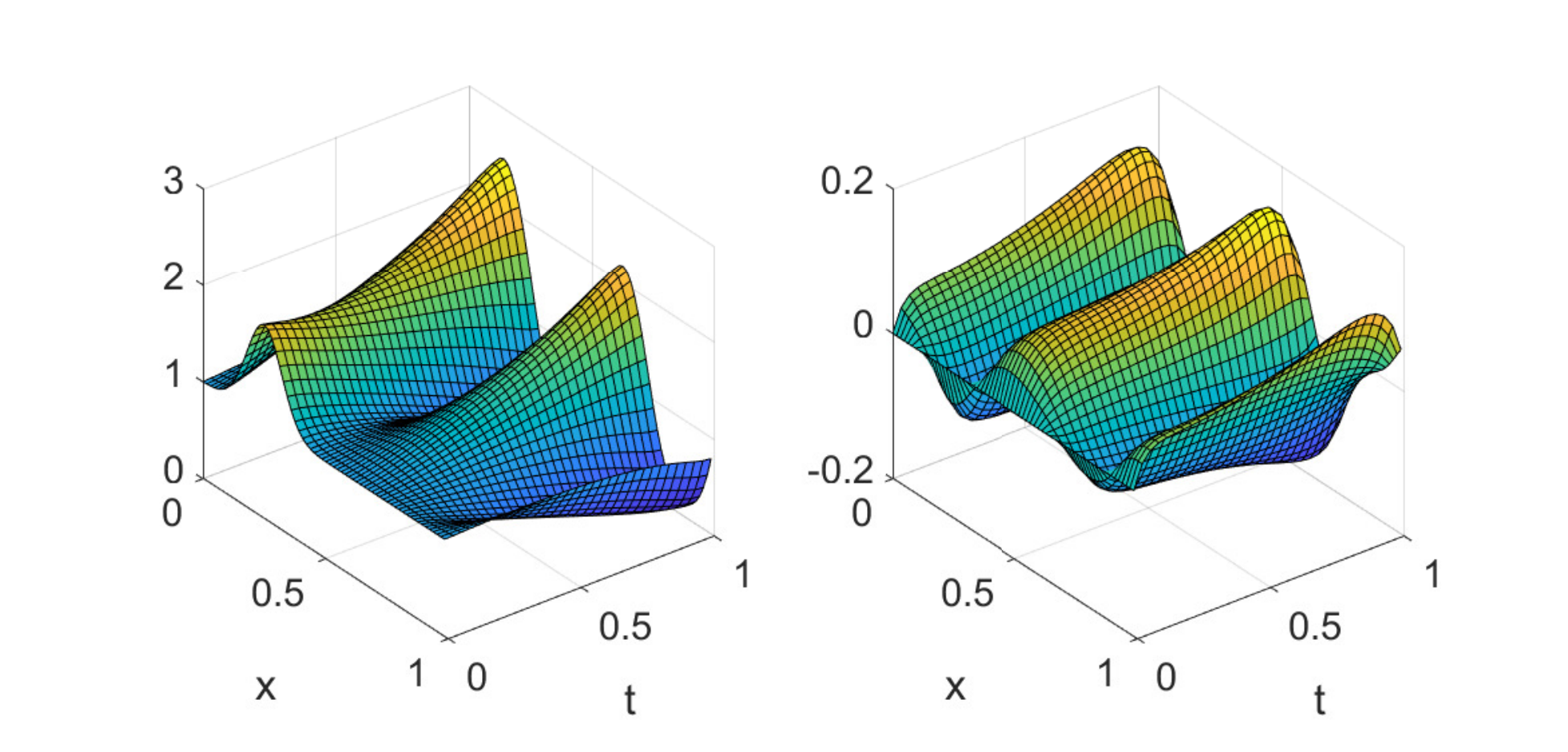} 
        \caption{The density (left) and the momentum (right) change over time for case $c_F = 2$ in example 3.}
                 \label{fig:eg2b}
\end{figure}
Figure \ref{fig:eg2a}, \ref{fig:eg2b} present the density and momentum profile for the control problems. The density forms a similar shape both cases, with density concentrate more around $x=\frac{3}{8}, \frac{7}{8}$. As for the momentum, the momentum in the second case $c_F = 2$ has a smaller magnitude in terms of $\max_{x,t} m(x,t)$.

\bibliographystyle{abbrv}

\end{document}